\documentclass[11pt,reqno,oneside]{amsart}
\usepackage{verbatim,amsmath,amssymb,cite,epsfig,xspace,xcolor}
\usepackage{srcltx}
\numberwithin{equation}{section}

\theoremstyle{plain}
\newtheorem{theorem}{Theorem}[section]
\newtheorem{lemma}{Lemma}[section]
\newtheorem{corollary}{Corollary}[section]
\theoremstyle{definition}

\theoremstyle{remark}
\newtheorem{remark}{Remark}[section]

\newcommand{\Real}{\mathbb R}

\newcommand{\A}{{\mathcal A}}

\newcommand{\F}{{\mathcal F}}
\newcommand{\C}{{\mathcal C}}

\newcommand{\I}{{\mathcal I}}
\newcommand{\J}{{\mathcal J}}

\newcommand{\kk}{k}
\newcommand{\half}{\frac{1}{2}}

\DeclareMathOperator*{\argmin}{argmin}
\newcommand{\eps}{\epsilon}
\newcommand{\ignore}[1]{}

\newcommand{\U}[2]{\mathcal{U}^{{#1},{#2}}}

\newcommand{\Ump}{\U{-1}{p}}
\newcommand{\bqc}{\Phi_{\beta}}
\newcommand{\nl}{\Phi}

\newcommand{\Y}{\mathcal{Y}_F}
\newcommand{\norm}[1]{\|#1\|}
\newcommand{\umpnorm}[1]{\norm{#1}_{\Ump}}
\newcommand{\bignorm}[1]{\left \|#1\right \|}
\newcommand{\Cone}{\bar{C}_1}
\newcommand{\Ctwo}{\bar{C}_2}
\newcommand{\Cthr}{\bar{C}_3}
\newcommand{\Cfour}{\bar{C}_4}

\newcommand{\lpe}{{\ell^p_{\eps}}}
\newcommand{\lqe}{{\ell^q_{\eps}}}
\newcommand{\lte}{{\ell^2_\eps}}
\newcommand{\p}{\phi}
\newcommand{\bbar}{\bar{\beta}}
\newcommand{\abar}{\bar{A}}
\newcommand{\oc}{\mathcal{O}}
\newcommand{\stabbound}{2\eps\Cthr\norm{\Delta\beta_\xi y''_\xi}_{\ell^\infty} +2 \eps^2 \{ \Cthr\norm{(1-\beta_{\xi-1})y'''_\xi}_{\ell^\infty} +\Cfour\norm{(1-\beta_{\xi})(y''_\xi)^2}_{\ell^\infty} \}}
\newcommand{\minusstabbound}{2\eps\Cthr\norm{\Delta\beta_\xi y''_\xi}_{\ell^\infty} -2 \eps^2 \{ \Cthr\norm{(1-\beta_{\xi-1})y'''_\xi}_{\ell^\infty} +\Cfour\norm{(1-\beta_{\xi})(y''_\xi)^2}_{\ell^\infty} \}}
\newcommand{\periodsum}{\sum_{\xi = 1}^N}

\newcommand{\ltwoconsbound}[2]{#1 \Ctwo\norm{ \Delta \beta_\xi y''_\xi}_{\lte} + #2 \{\Ctwo\norm{(1-\beta_{\xi-1})y'''_\xi}_{\lte} +\Cthr\norm{(1-\beta_{\xi})(y''_\xi)^2}_{\lte}\}}
\newcommand{\axi}{\alpha_{\xi}}
\newcommand{\bqce}{\Phi^{qce}_{\gamma}}
\newcommand{\bqnl}{\Phi^{qnl}_{\beta}}
\newcommand{\bqcab}{\Phi_{\alpha, \beta}}
\newcommand{\btotal}{\Phi^{total}_\beta}
\newcommand{\gtotal}{\Phi^{total}_\gamma}
\newcommand{\fv}{\delta}
\newcommand{\bigabs}[1]{\left | #1 \right|}
\newcommand{\gs}{\gamma_s}
\newcommand{\tk}{\eta}

\definecolor{purplecol}{RGB}{130, 0, 130}

\begin{document}

\title
[Analysis of Energy-Based Blended Quasicontinuum Approximations]
{Analysis of Energy-Based Blended Quasicontinuum Approximations}
\author{Brian Van Koten}
\author{Mitchell Luskin}

\address{Brian Van Koten\\
School of Mathematics \\
University of Minnesota \\
206 Church Street SE \\
Minneapolis, MN 55455 \\
U.S.A.}
\email{vank0068@umn.edu}
\address{Mitchell Luskin \\
School of Mathematics \\
University of Minnesota \\
206 Church Street SE \\
Minneapolis, MN 55455 \\
U.S.A.}
\email{luskin@umn.edu}

\thanks{
This work was supported in part by DMS-0757355,
 DMS-0811039,  the PIRE Grant OISE-0967140, the Institute for Mathematics and
Its Applications, and
 the University of Minnesota Supercomputing Institute.
  This work was also supported by the Department of Energy under Award Number
DE-SC0002085.
}



\date{\today}

\begin{abstract}
The development of patch test consistent quasicontinuum energies
for multi-dimensional crystalline solids modeled by many-body
potentials remains a challenge. The original quasicontinuum energy (QCE)~\cite{Miller:2008}
has been implemented for many-body potentials in two and three
space dimensions, but it is not patch test consistent. We propose that by blending
the atomistic and corresponding Cauchy-Born continuum models of QCE in an interfacial
region with thickness of a small number $k$ of blended atoms, a
general quasicontinuum energy (BQCE) can be developed with the potential to
significantly improve the accuracy of QCE
near lattice instabilities such as dislocation
formation and motion.

In this paper, we give an error analysis of the blended quasicontinuum
energy (BQCE) for a periodic one-dimensional chain of atoms with
next-nearest neighbor interactions.  Our analysis includes the optimization
of the blending
function for an improved convergence rate.
We show that the $\ell^2$ strain error
for the non-blended QCE energy (QCE), which has low order $\text{O}(\varepsilon^{1/2})$
where $\varepsilon$ is the atomistic length scale~\cite{Dobson:2008b,mingyang},
 can be reduced by a factor of $k^{3/2}$ for an optimized blending function where $k$ is the number of atoms in the blending region.
The QCE energy has been further shown to suffer from a O$(1)$ error in the critical strain at which the lattice
loses stability~\cite{doblusort:qce.stab}.
We prove that the error in the critical strain of BQCE
can be reduced by a factor of $k^2$ for an optimized blending function, thus demonstrating that the BQCE energy
for an optimized blending function has
the potential to give an
accurate approximation of the deformation near lattice instabilities such as crack growth.
\end{abstract}

\maketitle

\section{Introduction}\label{sec: introduction}

Crystalline materials often have highly singular strain fields
at crack tips, dislocations, and grain boundaries that require
the accuracy of atomistic modeling only in small regions
surrounding these defects.  However, these localized defects
interact through long-ranged elastic fields with a much larger region
where the strain gradients are sufficiently small to allow
accurate and efficient approximation by coarse-grained continuum finite element models.
This has motivated the development of numerical
methods that couple atomistic regions
with continuum regions to compute problems with
length scales that are sufficiently large for accurate
and reliable scientific and engineering application
~\cite{makridakis10,Legoll:2005,gaviniorbital,LinP:2003a,LinP:2006a,Ortner:2008a,E:2006,Dobson:2008a,Gunzburger:2008a,Miller:2008,PrudhommeBaumanOden:2006,shapeev}.

The quasicontinuum
energy (QCE) ~\cite{Ortiz:1995a} couples
an atomistic model to a finite element continuum model based on the
Cauchy-Born strain energy density.  The Cauchy-Born strain energy density
reproduces the atomistic energy density for a uniformly deformed lattice.
The QCE atomistic-to-continuum coupling also reproduces the atomistic
energy density for a uniformly deformed lattice simply by modifying the volume
of the triangles or tetrahedra within the cut-off radius of the boundary of
the atomistic region by using the Voronoi cell~\cite{Miller:2008}.  The QCE energy remains the most
popular quasicontinuum energy since the modification of the volume of the
triangles near the interface to reproduce the atomistic energy density
can be explicitly achieved and has been implemented in the code~\cite{qcmethod}
 for multi-dimensional problems with many-body potentials.

A quasicontinuum energy is said to be {\em patch test consistent}
if it reproduces the zero net forces at each atom for a uniformly
deformed lattice.  The fully atomistic and Cauchy-Born energies
are patch test consistent by symmetry, but the symmetry is broken in
the QCE atomistic-to-continuum interface.  The nonzero forces
in the QCE atomistic-to-continuum interface of a uniformly
deformed lattice are called {\em ghost forces}~\cite{dobs-qcf2,dobsonluskin07,Miller:2008,rodney_gf}.

The {\em ghost force correction method} (GFC) was developed to
improve the accuracy of QCE~\cite{rodney_gf}, and it has been shown
to converge to the patch test consistent force-based quasicontinuum method (QCF)~\cite{curtin_miller_coupling,dobsonluskin07}.
GFC has been implemented in the code~\cite{qcmethod} and further developed to
utilize the efficiency
of atomistic-to-continuum modeling and continuum mesh {\em a posteriori}
 adaptivity~\cite{arndtluskin07c,arndtluskin07b,Miller:2008,rodney_gf}.

However, it has been shown that the
force-based quasicontinuum method (QCF) gives a non-conservative
force field~\cite{curtin_miller_coupling,dobsonluskin07,Miller:2008},
so the stability of a deformation that satisfies the QCF equations
cannot be simply determined by checking whether it is a local minimum
of a quasicontinuum energy~\cite{doblusort:qcf.stab}.
The linearization of the QCF method has also been shown to be indefinite
and to have unusual stability properties ~\cite{dobs-qcf2} which presents a further challenge
to the notion of stability and
to the development of efficient and reliable iterative methods~\cite{qcf.iterative,luskin.iter.stat}.

The potential for more reliable and efficient iterative solution methods and more direct extensions
to finite temperature dynamical methods~\cite{PhysRevLett.95.060202} make energy-based quasicontinuum methods more desirable
than force-based quasicontinuum methods.  Several quasicontinuum energies have been proposed which
satisfy the patch test for a limited class of problems, but a quasicontinuum energy that satisfies the
patch test for many-body potential interactions in several space dimensions with general coupling interfaces
and coarsening has yet to be developed.

The quasi-nonlocal energy (QNL)~\cite{Shimokawa:2004} gives an explicit
and implemented algorithm for
close range interactions (up to next-nearest neighbor interactions for a chain
and up to fourth-nearest neighbor interactions for a FCC lattice) by restoring the
symmetry of the interactions.  Ghost forces remain for nonplanar interfaces
and coarsening~\cite{E:2006}.  The geometric consistency approach gives geometric
and algebraic conditions to
allow finite range interactions, but an implemented algorithm has only been given
for low index planar interfaces.  An efficient and implemented algorithm
has yet to be given for a general nonplanar interface
surrounding a defect.
Ghost forces still remain for nonplanar interfaces
and coarsening~\cite{E:2006}.  Further, the generalization of the
geometric consistency approach to three-dimensional problems seems to depend on special
decompositions of three-dimensional space into congruent tetrahedra (for example, the hexagonal
lattice is such a special decomposition of two-dimensional space).

The QCE energy, on the other hand, can be implemented for any decomposition of
three-dimensional space into tetrahedra (see (21) in ~\cite{Miller:2008}),
and so can the blended quasicontinuum energy (BQCE) which we propose below.
In fact, a code for the BQCE energy can be obtained by
simply modifying any QCE code and suppressing the coarsening
in the blending interface. We have done that
for our two-dimensional QCE code~\cite{luskin.durham}. The BQCE code will
then be able to utilize any adaptive atomistic-to-continuum
and continuum coarsening features of the QCE code.

More recently, generalizations have been given for the explicit extension of the QNL energy
to allow finite range interactions~\cite{shapeev,LuskinXingjie}.  The consistent a/c
coupling ~\cite{shapeev} gives an implemented patch test consistent quasicontinuum energy for two-dimensional
problems with general pair potential interactions, atomistic-to-continuum coupling interface,
and coarsening.
No generalizations of these or other patch test consistent methods to three-dimensional crystals, multi-lattices, or multi-body potentials are
yet known.

Since there remain important problems for which no patch test consistent methods are known,
we feel it is important to develop general strategies for reducing the error of coupled
energies which can be applied to a broad class of methods and problems.
We propose that by blending
the atomistic and corresponding Cauchy-Born continuum energies of QCE in an interfacial
region with thickness  of a small number $k$ of blended atoms, a blended
quasicontinuum energy (BQCE) can be developed with the potential to
significantly improve the accuracy of the QCE energy
near lattice instabilities such as dislocation
formation and motion.
This blended quasicontinuum energy can be implemented for any problem for which a QCE energy has been developed,
and we think that most existing codes for QCE could be easily modified to implement BQCE.

The blended quasicontinuum energy is not patch test consistent,
but our error estimates in Section~\ref{subsec: modeling}
show that the error due to the ghost force can be reduced significantly by optimizing the blending function
over a blending region of small size $k.$
Therefore, we think that the use of BQCE is a good strategy for reducing the error of QCE when applied to
problems for which no patch test consistent methods are known.
Moreover, patch test consistent quasicontinuum methods tend to be more complicated and more difficult to implement than QCE.
Thus, even for problems for where a patch test consistent quasicontinuum energy is known,
we think that the BQCE energy remains an attractive quasicontinuum energy.

The BQCE method is similar to many other methods in which atomistic and continuum energies are smoothly blended over an interfacial
region, with some features specific to the quasicontinuum setting.
We call such a region a \emph{blending region}, and we call such a method a \emph{blended} method.
(Other authors call the blending region a ``handshake region", an ``interface", a ``bridging region", or an ``overlap".)
We call the weights which blend the atomistic and continuum contributions to the energy \emph{blending functions}.
For example, the AtC coupling~\cite{badia:forcebasedAtCcoupling},
the bridging domain method~\cite{xiao:bridgingdomain},
and the Arlequin method for coupling particle and continuum models~\cite{bauman:applicationofArlequin}
are all energy-based blended methods which are similar to BQCE.
By contrast, in other schemes, the atomistic and continuum models
are coupled abruptly across a sharp interface consisting of only a few atoms.
Such schemes include the energy-based quasicontinuum (QCE) method~\cite{Ortiz:1995a},
the quasinonlocal quasicontinuum (QNL) method~\cite{Shimokawa:2004},
the generalized QNL method~\cite{LuskinXingjie},
and the consistent a/c coupling~\cite{shapeev}.

In this paper, we present an error analysis of the blended quasicontinuum energy (BQCE).
We determine precisely how the error of the BQCE method depends on the size of the blending region
and how it depends on the blending functions.
We analyze the accuracy of BQCE applied to the problem of a one-dimensional chain of atoms
with periodic boundary conditions and a next nearest neighbor pair interaction model.
Our choice of a one-dimensional analysis allows us to explicitly investigate the accuracy of the
BQCE energy for deformations up to lattice instability~\cite{doblusort:qce.stab},
which is crucial for the computation of critical strains in
lattice statics as well as for the computation of the dynamics of defect
nucleation and movement.

We focus on modeling error and do not consider coarse-graining
in our analysis. Instead, we assume that the continuum energy
is a Cauchy-Born energy discretized by piecewise linear finite
elements with nodes at every atom. We refer the reader
to~\cite{Ortner:2008a,ortner.wang} for an analysis of
coarse-graining for a problem similar to our one-dimensional
chain. In addition, we constrain the displacements of atoms in
the blending region to exactly match the continuum displacement
field. In some other methods, the displacements of atoms in the
blending region are coupled weakly to the continuum
displacement field using Lagrange multipliers or a penalty
method. This is the case in the Arlequin
method~\cite{bauman:applicationofArlequin}, in the bridging
domain method~\cite{xiao:bridgingdomain}, and in some
implementations of the AtC
Coupling~\cite{badia:forcebasedAtCcoupling,
badia:onAtCcouplingbyblending}. Computational experiments
comparing various methods of weak coupling are given in
~\cite{seleson:bridgingmethods,badia:onAtCcouplingbyblending,badia:forcebasedAtCcoupling,xiao:bridgingdomain}.
A numerical analysis of a mixed finite element method for a
weakly coupled problem is given
in~\cite{bauman:applicationofArlequin}.

Blending can also improve the accuracy of quasicontinuum
energies which are more accurate than QCE, but may not satisfy
the patch test for general interactions, nonplanar interfaces,
or coarsening.  We thus also define and analyze a blended
generalization of QNL which we call the \emph{blended QNL}
(BQNL) energy. The BQNL method is an energy-based, blended
coupling which passes the patch test for the one-dimensional
problem analyzed in this paper, but does not generally satisfy
the patch test as described above.

Our analysis extends the techniques developed
in~\cite{ortner:qnl1d} for the patch test consistent QNL to
obtain precise estimates of the error for BQCE, which does not
satisfy the patch test. Most importantly, our estimates can be
used to optimize the blending function and blending interface
size of BQCE to obtain accurate solutions up to lattice
instabilities. Our error estimate in Theorem~\ref{thm: a priori
existence for BQCE} shows that the $\ell^2$ strain error for
the BQCE method can be reduced by a factor of $k^{3/2}$ where
$k$ is the number of atoms in the blending region. This result
is suggested by the results of~\cite{Dobson:2008b,mingyang} on
the decay of the coupling error for QCE. It has been shown that
the QCE method suffers from a O$(1)$ error in the critical
strain at which the lattice loses stability (which models
fracture or the formation of a
defect)~\cite{doblusort:qce.stab}. We prove in
Theorem~\ref{thm: aprstab} that the error in the critical
strain of BQCE can be reduced by a factor of $k^2$ where $k$ is
the number of atoms in the blended interface region. In
Remark~\ref{rem: ghost force}, we use our modeling error
estimates to derive blending functions for the BQCE method
which minimize the error due to the ghost force. We give our
\emph{a priori} error estimate for BQNL in Theorem~\ref{thm: a
priori existence for BQNL}. In Theorem~\ref{thm: a priori
existence for BQNL}, we show that the $\ell^2$ strain error for
the BQNL method can be reduced by a factor of $k^{1/2}$. We
prove in Theorem~\ref{thm: aprstab} that the BQNL method does
not suffer from a critical strain error.

\section{An atomistic model and its quasicontinuum approximations}\label{sec: model}

\subsection{The atomistic model problem}\label{subsec: atomistic model}
We begin by presenting a model of a one-dimensional lattice of atoms.
Our reference configuration is the integer lattice $\eps \mathbb{Z}$ with lattice spacing $\eps > 0$.
The admissible deformations are N-periodic, mean-zero displacements of uniformly strained states of $\eps \mathbb{Z}$.

Precisely, for $N \in \mathbb{N}$, we let
\begin{equation*}
\mathcal{U} := \{ u: \mathbb{Z} \rightarrow \Real \mbox{ : }
u_\xi = u_{\xi+N} \mbox{ for all } \xi \in \mathbb{Z} \mbox{
and } \periodsum u_\xi = 0\}
\end{equation*}
be the set of N-periodic, mean-zero displacements.
For $F \in (0,\infty)$ we let
\begin{equation*}
y^F_\xi := F \eps \xi
\end{equation*}
be the uniformly deformed state with \emph{macroscopic strain} $F$,
and we let
\begin{equation*}
\Y := \{ y: \mathbb{Z} \rightarrow \Real \mbox{ : } y_\xi = y^F_\xi + u_\xi \mbox{ for some } u \in \mathcal{U} \}
\end{equation*}
be the set of admissible deformations with fixed macroscopic
strain $F$. We will fix $\eps = \frac{1}{N}$ throughout the
remainder of the paper, so that the reference length of a
period is one, independently of $\eps$. For notational
convenience, we have let the deformations $y \in \Y$ be
functions of $\mathbb{Z}$ instead of $\eps \mathbb{Z}$ even
though we still think of $\eps \mathbb{Z}$ as the reference
configuration. We also remark that the ``displacements" $u$
which appear in the definition of $\Y$ are not displacements
measured relative to the reference lattice $\eps \mathbb{Z}$,
but are instead displacements relative to the uniformly
deformed state $y^F$.

In the language of mechanics, one would say that our
choice of the deformation space $\Y$ means that we have imposed
``periodic boundary conditions with macroscopic strain $F$."
Periodic boundary conditions are mathematically similar to
Dirichlet boundary conditions in that constraints are applied
to the space of admissible deformations. In our case, the
constraint is given by the macroscopic strain $F$ which models
the macroscopic scale stretching or contraction of the chain.
Periodic boundary conditions are useful because they eliminate
surface effects near the boundaries of the chain.

For a deformation $y \in \Y,$ we define a \emph{stored energy per
  period} which we call $\nl(y)$. Our stored energy $\nl(y)$ sums the
energies of all the interactions between first nearest neighbors and
second nearest neighbors in the lattice. We compute the energies of
interaction using a potential $\phi:(0,\infty]\rightarrow\Real$ which
we assume satisfies
\begin{enumerate}
\item $\phi \in C^4((0,\infty], \Real);$
\item there exists $r^*$ so that $\p''(r) > 0$ for all $r
    \in (0, r^*)$, and $\p''(r) < 0$ for all $r \in
    (r^*,\infty);$
\item
$\phi$ and its derivatives decay at infinity.
\end{enumerate}
Examples of commonly used interaction potentials which satisfy the
above conditions include the Lennard-Jones potential and the Morse
potential.
Specifically, we define
\begin{equation}\label{def: nl energy}
\nl(y) = \eps \periodsum \p(y'_\xi) + \p(y'_\xi + y'_{\xi+1}),\quad \mbox{ where}\quad
y'_\xi := \frac{y_\xi -y_{\xi-1}}{\eps}.
\end{equation}
In~\eqref{def: nl energy}, we call the terms of the form
$\p(y'_\xi)$ \emph{first nearest neighbor interactions}, and we
call terms of the form $\p(y'_\xi + y'_{\xi+1})$ \emph{second
nearest neighbor interactions}.

 The reader will observe that we have scaled the energy
by $\eps$ and the interatomic distances by $\eps^{-1}$. This is
done for two related reasons. First, quasicontinuum methods are
only needed when when $N$ is large (equivalently, when $\eps$
is small). By introducing the scaling above, we are able to
distinguish those parts of the error which are most significant
in the practically relevant limit of large $N$. The reader
should bear this in mind when interpreting our error estimates
in Sections~\ref{subsec: modeling},~\ref{subsec: stability},
and~\ref{subsec: a priori existence}. Second, we desire that in
the limit as $\eps$ tends to zero (equivalently, as $N$ tends
to infinity), the atomistic energy converges to the fully
continuum Cauchy-Born energy~\eqref{eqn: fully continuum cauchy
born energy}. This is explained in detail in
Section~\ref{subsec: cauchy born approx} below. Roughly
speaking, this ensures that the atomistic energy is compatible
with the continuum energy to which it is coupled in the BQCE
and BQNL schemes. The compatibility of the two energies is
reflected in the fact that the modeling error of the
Cauchy-Born energy~\eqref{eqn: cauchy born rule} is of order
$\eps^2$ (See Remark~\ref{rem: clarification of modeling}).

Now suppose we want to model the response of our lattice to a dead load.
We let $\mathcal{U}^*$  denote the algebraic dual of $\mathcal{U}$,
and we equip the space $\mathcal{U}$ with the \emph{$\ell^2_\eps$ inner product} given by
\begin{equation}\label{eqn: inner product}
<u, v> := \eps \periodsum u_\xi v_\xi.
\end{equation}
Generally speaking, a dead load $f$ should be thought of as an element of $\mathcal{U}^*$.
However, since $\mathcal{U}$ is a Hilbert space with the inner product~\eqref{eqn: inner product},
any $g \in \mathcal{U}^*$ is represented by some $f \in \mathcal{U}$.
That is, for any $g \in \mathcal{U}^*$ there exists $f \in \mathcal{U}$ so that for all $v \in \mathcal{U}$,
$g(v) = <f, v>$.
Thus, we will think of dead loads as elements of $\mathcal{U}$.

The total energy for an atomic chain subject to a dead load $f \in \mathcal{U}$ is then given by
\begin{equation*}
\Phi^{total}(y) = \nl(y) - <f,y>,
\end{equation*}
and the equilibrium deformation of the atoms solves the minimization
problem:
\begin{equation}\label{eqn: minimization problem}
y \in \argmin_{y \in \Y} \Phi^{total}.
\end{equation}
Let $\fv \Phi$ denote the first variation of the energy $\Phi$, and let
$\fv^2 \Phi$ denote the second variation.
We will say that a solution $y \in \Y$ of the minimization problem~\eqref{eqn: minimization problem}
is \emph{strongly stable} if $\fv^2 \nl(y)[u,u]>0$ for all $u \in \mathcal{U} \setminus \{0\}$,
and we will call a deformation $y \in
\Y$ an \emph{equilibrium} if it solves
\begin{equation*}
\fv \Phi (y)[u] = <f,y> \mbox{ for all } u \in \mathcal{U}.
\end{equation*}

\subsection{The Cauchy-Born approximation}\label{subsec: cauchy born approx}
We call an energy $\Psi(y)$ \emph{local} if it can be written in the form
\begin{equation*}
\Psi(y) = \eps \periodsum \psi(y'_\xi)
\end{equation*}
where $\psi: (0, \infty] \rightarrow \Real$ is called a
\emph{strain energy density}. Otherwise, we say that an energy
is \emph{nonlocal}. Observe that energy~\eqref{def: nl energy}
is nonlocal. The development and use of a strain energy
density $\phi$ allows coarse-graining in the continuum region
by utilizing the full range of algorithms, codes, and analysis
developed for the continuum finite element method. (An
alternative quasicontinuum approach uses quadrature to
approximate the nonlocal coarse-grained energy without the
construction of a strain energy
density~\cite{Gunzburger:2008a}. Peridynamics offer another
promising fully nonlocal approach to
coarse-graining~\cite{peri}.) Thus, it is useful to devise
local energies which approximate~\eqref{def: nl energy}. The
key to developing such energies is to observe that as long as
$y'_\xi$ varies slowly between neighboring lattice points, we
have
\begin{equation*}
y'_{\xi} + y'_{\xi+1} \approx 2y'_{\xi} \approx 2y'_{\xi+1}.
\end{equation*}
If we replace the second nearest neighbor terms $\p(y'_{\xi} + y'_{\xi+1})$ by $\half \{ \p(2y'_{\xi}) + \p(2y'_{\xi+1}) \}$
in~\eqref{def: nl energy}, we then obtain the \emph{Cauchy-Born energy}
\begin{equation}\label{eqn: cauchy born rule}
\Phi^{cb}(y) := \eps \periodsum \p(y'_\xi) + \half \{ \p(2y'_{\xi}) + \p(2y'_{\xi+1}) \},
\end{equation}
which is a local energy commonly used to approximate~\eqref{def: nl energy}.

An alternative derivation of the Cauchy-Born energy is as follows.
Suppose that $Y : [0,1] \rightarrow \Real$ is a $C^\infty$ deformation of $[0,1]$,
and define $y^\eps_\xi = Y(\eps \xi)$.
It can be shown that
\begin{equation*}
\lim_{\eps \rightarrow 0^+} \Phi(y^\eps) = \int_{[0,1]} \psi^{cb}\left (\frac{dY}{dx}(x) \right ) dx =: \Psi^{cb}(Y),
\end{equation*}
where $\psi^{cb}(r) := \p(r) + \p(2r)$ is called the \emph{Cauchy-Born strain energy density}.
We call the energy $\Psi^{cb}$ the \emph{fully continuum Cauchy-Born energy}.
We refer the reader to~\cite{lions} for a detailed discussion of the convergence of $\Phi(y^\eps)$  to $\Psi^{cb}(Y)$ as $\eps$ goes to zero.

Now suppose that $Y$ is in the Lagrange $P^1$ finite element
space on $[0,1]$ with nodes at the reference position of each
atom; so the nodes are $\eps \xi$ for $\xi = 1, \dots, N$. In
that case, it is easy to show that
\begin{equation}\label{eqn: fully continuum cauchy born energy}
\Psi^{cb}(Y) = \Phi^{cb}(y^\eps).
\end{equation}
Thus, we see that the Cauchy-Born energy $\Phi^{cb}$ is a
finite element approximation of the fully continuum Cauchy-Born
energy $\Psi^{cb}$. Consequently, we will refer to $\Phi^{cb}$
as a ``continuum" energy. We say that the energy $\Phi^{cb}$ is
not coarse-grained since $\Phi^{cb}$ has the same number of
degrees of freedom as the atomistic energy $\nl$: both depend
on the deformed position of each atom.

\subsection{Quasicontinuum approximations}\label{subsec: blended quasicontinuum approx}

Suppose we want to approximate a solution to the minimization problem~\eqref{eqn: minimization problem}
which has defects in localized regions but which is smooth throughout most of the lattice.
Efficiency requires that most of the lattice be modeled using the Cauchy-Born energy~\eqref{eqn: cauchy born rule},
while accuracy requires that the defects be treated using the nonlocal energy~\eqref{def: nl energy}.
Thus, we desire a coupling of the two models.

We distinguish two approaches to deriving coupled energies.
We call these the \emph{strain-energy based} approach and \emph{bond based} approach.
In the strain-energy based approach,
one first defines an atomistic \emph{energy per atom}.
For our one-dimensional chain, an appropriate energy per atom is
\begin{equation}\label{eqn: atomistic energy per atom}
\Phi^a_\xi (y) := \half \{\p(y'_\xi)  + \p(y'_{\xi+1}) + \p(y'_\xi + y'_{\xi-1}) + \p(y'_{\xi+2} + y'_{\xi+1}) \}.
\end{equation}
We call $\Phi^a_\xi(y)$ the \emph{atomistic energy at atom $\xi$}.
Observe that $\Phi^a_\xi(y)$ is half the total energy of all bonds involving atom $\xi$,
so $\Phi(y) = \eps \periodsum \Phi^a_\xi(y)$.
The energy per atom should be interpreted as the energy per length $\eps$ at atom $\xi$;
thus, $\Phi_\xi(y)$ is analogous to a continuum strain-energy density.

To define a coupled energy we now blend the atomistic energy per atom with the continuum energy.
Schematically, we choose some blending functions $\alpha:[0,1] \rightarrow [0,1]$ and $\beta:\eps \mathbb{Z} \rightarrow [0,1]$.
Then for a deformation $Y:[0,1] \rightarrow \Real$ with corresponding atomistic deformation $y^\eps_\xi := Y(\eps \xi)$,
we define
\begin{equation}\label{eqn: schematic coupled energy}
\Phi^{coupled}(Y) := \int_{[0,1]} \alpha(x) \psi^{cb} \left (\frac{dY}{dx}(x) \right ) dx + \eps \periodsum \beta_\xi \Phi^a_\xi(y^\eps).
\end{equation}
In equation~\eqref{eqn: schematic coupled energy},
one should imagine that the function $\alpha$ is zero in a small region containing any defects,
and one throughout the bulk of the lattice.
One should imagine that $\beta$ is one near defects,
and zero throughout the bulk of the lattice.
For an abruptly coupled energy, one should imagine that $\alpha$ is the characteristic function of the continuum region,
and $\beta$ is the characteristic function of the atomistic region.
By contrast, for a blended energy, one should imagine that $\alpha$ and $\beta$ are more smooth.
The energy-based quasicontinuum (QCE) method~\cite{Miller:2008}, the bridging domain method~\cite{xiao:bridgingdomain},
and the Arlequin method~\cite{bauman:applicationofArlequin} all take essentially the form~\eqref{eqn: schematic coupled energy}.
Moreover, although the AtC method~\cite{badia:forcebasedAtCcoupling} was derived in variational form,
the equilibrium condition is the Euler-Lagrange equation of an energy similar to~\eqref{eqn: schematic coupled energy}.

We will now explain the QCE energy in detail.
Let $[0,1] = \mathcal{A} \cup \mathcal{C}$ be a partition of the domain $[0,1]$ into an atomistic region $\mathcal{A}$
and a continuum region $\mathcal{C}$.
Suppose that $Y$ is in the $P^1$ Lagrange finite element space with nodes at every lattice site,
and let $y_\xi := Y(\eps \xi)$ be the corresponding deformation of the atomistic lattice $\eps \mathbb{Z}$.
In this special case,
the QCE energy reduces to
\begin{equation}\label{eqn: qce energy}
\begin{split}
\Phi^{qce}(y) :&= \eps \sum_{\eps \xi \in \mathcal{A}} \Phi^a_\xi (y) + \eps \sum_{\eps \xi \in \mathcal{C}} \Phi^c_\xi (y), \mbox{ where } \\
\Phi^{c}_\xi(y) :&= \half \{\p(y'_\xi)  + \p(y'_{\xi+1}) + \p(2 y'_\xi) + \p(2y'_{\xi+1}) \}.
\end{split}
\end{equation}
We call $\Phi^{c}_\xi$ the \emph{continuum energy at} $\xi$.

Our blended energy-based quasicontinuum energy (BQCE) is based on a blend of the
atomistic energy at $\xi$~\eqref{eqn: atomistic energy per atom}
and the continuum energy at $\xi$~\eqref{eqn: qce energy}.
Let $\gamma: \mathbb{Z} \rightarrow [0,1]$ be a blending function.
For $y \in \Y$ we define
\begin{equation}
\begin{split}
\Phi^{bqce}_\gamma(y) :&= \eps \periodsum \gamma_\xi \Phi^{c}_\xi(y) + (1 - \gamma_\xi) \Phi^a_\xi(y) \\
&= \eps \periodsum \p(y'_{\xi}) + \frac{\gamma_\xi}{2}
 \{\p(2y'_\xi)+\p(2y'_{\xi+1})\}
       + \frac{(1-\gamma_\xi)}{2} \{\p(y'_{\xi-1} + y'_{\xi})+\p(y'_{\xi+1} + y'_{\xi+2})\}.
\end{split}
\end{equation}
\ignore{
In general, we will assume that $\gamma$ is zero in a small region containing any defects,
and that $\gamma$ equals one throughout the bulk of the lattice.
We will call the set of all $\xi$ such that $\gamma_\xi = 0$ the \emph{atomistic region},
we will call the set of $\xi$ such that $\gamma_\xi = 1$ the \emph{continuum},
and we will call the set of all other $\xi$ the \emph{interface}.}
We observe that the QCE energy is simply the BQCE energy whose blending function $\gamma$
is the characteristic function of the continuum region.
We will show in Theorem~\ref{thm: modeling} that BQCE does not pass the patch test where we recall that
a quasicontinuum energy passes the \emph{patch test} if there are no forces under uniform strain;
for our one-dimensional problem an energy $\Phi^{coupled}$ passes that patch test if $\fv \Phi^{coupled}(y^F) = 0$ for all $F \in (0, \infty)$.
We will call a method which passes the patch test \emph{patch test consistent}.

We will now discuss the QNL and BQNL energies in detail.
The atomistic and continuum energies of the bond between the nearest neighbors at reference positions $\eps \xi$ and $\eps (\xi - 1)$
are taken to be $\phi(y'_\xi)$.
This is the same as the energy of that bond in the fully atomistic model~\eqref{def: nl energy}.
The atomistic energy of the bond between the second nearest neighbors at reference positions $\eps (\xi + 1)$ and $\eps (\xi - 1)$
is taken to be $\phi(y'_{\xi+1} + y'_\xi)$,
and the continuum energy of that bond is taken to be $\half \{\phi(y'_{\xi+1}) + \phi(y'_\xi)\}$.
As discussed in Section~\ref{subsec: cauchy born approx}, this choice of continuum bond energy is related to the Cauchy-Born rule.
The coupled energy is then given by
\begin{equation}\label{eqn: qnl energy}
\Phi^{qnl}(y) := \eps \periodsum \p(y'_\xi) + \eps \sum_{\eps \xi \in \mathcal{A}} \phi(y'_{\xi+1} + y'_\xi)
 + \eps \sum_{\eps \xi \in \mathcal{C}} \half \{\phi(y'_{\xi+1}) + \phi(y'_\xi)\},
\end{equation}
where $[0,1] = \mathcal{A} \cup \mathcal{C}$ is a partition of $[0,1]$.

Our blended quasinonlocal QC (BQNL) energy is based on a smooth blending of the atomistic and continuum bond energies
used to define the QNL energy~\eqref{eqn: qnl energy}.
Let $\eta : \mathbb{Z} \rightarrow [0,1]$ be a blending function.
We define
\begin{equation}\label{eqn: bqnl energy}
  \bqnl(y) := \eps \periodsum \p(y'_{\xi}) + \eta_\xi \p(y'_{\xi} + y'_{\xi+1}) + \frac{1-\eta_\xi}{2} \{\p(2y'_{\xi}) + \p(2y'_{\xi+1}) \}.
\end{equation}
We remark that the QNL energy is the BQNL energy whose blending function is the characteristic function of the atomistic region $\mathcal{A}$.

We will now give a simple proof that BQNL passes the patch test.
We discuss the modeling error of BQNL in more detail in Section~\ref{subsec: modeling}.
The patch test consistency of BQNL is a consequence of the following result.
\begin{lemma}\label{lem: bqc is affine closure of qnl}
The set of BQNL energies is the affine hull of the set of QNL energies.
In particular, any BQNL energy may be expressed as an affine combination of QNL energies with different atomistic and continuum regions.
\end{lemma}
\begin{proof}
The reader may verify that the set of BQNL energies is an affine space.
Thus, to prove the Lemma it suffices to express every BQNL energy as an
affine linear combination of QNL energies with different atomistic and continuum regions.
Let $\Phi^{qnl}_\beta$ be the BQNL energy with blending function $\beta$,
and let $\Phi^{qnl}_{e_i}$ be the QNL energy with atomistic region $\A = \{i\}$
and continuum region $\C = \{1, \dots, N\} \setminus \{i\}$ for $i = 1, \dots, N$.
We compute
\begin{equation*}
\Phi^{qnl}_{e_i}(y) - \Phi^{cb}(y) = \eps \left ( \p(y'_i + y'_{i+1})-\half\{\p(2y'_i)+\p(2y'_{i+1})\} \right ),
\end{equation*}
and so we derive
\begin{equation*}
\begin{split}
\bqnl(y) - \Phi^{cb}(y) &= \eps \periodsum \beta_\xi \left (\p(y'_{\xi} + y'_{\xi+1})-\half\{\p(2y'_\xi)+\p(2y'_{\xi+1})\} \right) \\
                        &= \sum_{i=1}^N \beta_i (\Phi^{qnl}_{e_i}(y) - \Phi^{cb}(y))
                        = \sum_{i=1}^N \beta_i \Phi^{qnl}_{e_i}(y) - \Phi^{cb}(y)\sum_{i=1}^N \beta_i.
\end{split}
\end{equation*}
Then we have
\begin{equation*}
\bqnl(y) = \sum_{i=1}^N \beta_i \Phi^{qnl}_{e_i}(y) + (1-\sum_{i=1}^N \beta_i) \Phi^{cb}(y).
\end{equation*}
This expresses $\bqnl$ as an affine combination of QNL energies,
since we recall that the Cauchy-Born energy, $\Phi^{cb}$, is the QNL energy whose continuum region is the entire domain.
\end{proof}

In light Lemma~\ref{lem: bqc is affine closure of qnl},
one expects that BQNL will inherit many of the properties of the QNL energy.
In particular,
since any BQNL energy is an affine combination of QNL energies,
and since the QNL energy is patch test consistent~\cite{Shimokawa:2004},
the BQNL energy passes the patch test.
Moreover, we show in Remark~\ref{rem: accuracy of critical strain} that the BQNL energy predicts the critical strain of the atomistic energy as accurately as the QNL method.

\begin{remark}\label{rem: construction of consistent blended energies}
(Construction of Patch Test Consistent Blended Methods).
Lemma~\ref{lem: bqc is affine closure of qnl} suggests a general method for constructing patch test consistent, blended methods from patch test consistent methods with a sharp interface:
one can define a patch test consistent, blended method by taking a convex combination of patch test consistent energies with different atomistic and continuum regions.
Of course, it may be that the method so constructed is not practical.
Nevertheless, we feel that this observation could be useful in deriving patch test consistent, blended couplings.
\end{remark}

For our analysis, it will be convenient to define a single energy which incorporates both BQCE and BQNL as special cases.
We will use the
\emph{blended quasicontinuum (BQC) energy}
\begin{equation}\label{def: blended energy}
\bqcab(y) := \eps \periodsum \p(y'_{\xi}) + \axi \p (2 y'_{\xi}) + \beta_\xi \p (y'_{\xi} + y'_{\xi+1})
\end{equation}
where $\alpha, \beta: \mathbb{Z} \rightarrow [0,1]$ are
\emph{blending functions}.
We observe that the BQCE energy with blending function $\gamma$ is the same as the BQC energy with blending functions
\begin{equation} \label{eqn: alpha and beta for BQCE}
\alpha_\xi := \overline{\gamma_\xi} := \frac{\gamma_\xi + \gamma_{\xi -1}}{2} \quad\mbox{and}\quad
\beta_\xi := 1 - \frac{\gamma_{\xi+1} + \gamma_{\xi-1}}{2}.
\end{equation}
The BQNL energy with blending function $\eta$ is the BQC energy with blending functions
\begin{equation}\label{eqn: alpha and beta for BQNL}
\alpha_\xi := 1 - \overline{\eta_\xi} := 1- \frac{\eta_\xi + \eta_{\xi-1}}{2} \quad\mbox{and}\quad
\beta_\xi := \eta_\xi.
\end{equation}

\subsection{Summary of notation and auxiliary theorems}\label{subsec: notation}
Here we collect all the notation and auxiliary theorems which
we will use below. The reader may feel free to skim this
section at first and return only as necessary.

For differences, we write
\begin{alignat*}{2}
\Delta y_\xi :&= y_\xi - y_{\xi-1}, &
\Delta^2 y_\xi :&= y_{\xi+1} - 2 y_\xi + y_{\xi-1}, \\
y'_\xi :&= \frac{y_\xi - y_{\xi-1}}{\eps}, &
y''_\xi :&= \frac{y_{\xi+1} - 2y_\xi + y_{\xi-1}}{\eps^2},  \\
y'''_\xi :&= \frac{y_{\xi+1} - 3y_\xi + 3y_{\xi-1} - y_{\xi-2}}{\eps^3}.
\end{alignat*}
For means, we write
\begin{equation*}
\overline{y}_\xi := \frac{y_\xi + y_{\xi-1}}{2}.
\end{equation*}
For $y \in \Y$, we will think of $y'$, $y''$, $y'''$, and $\overline{y}$ as N-periodic functions defined on the reference configuration.

We will also use certain bounds on the potential function
$\phi$ and its derivatives:
\begin{equation*}
C_i(r_0) := \sup_{r \geq r_0} |\phi^{(i)} (r)| \mbox{ for } i = 1, \dots, 4.
\end{equation*}
We will denote the convex hull of the set $A$ by
\begin{equation*}
 \mbox{conv }A.
\end{equation*}

We define the following norms
\begin{alignat*}{2}
\norm{y}_{\lpe} :&= \left \{ \eps \periodsum |y_\xi|^p \right \}^{\frac{1}{p}} \mbox{ for } p \in [1,\infty), &\qquad
\norm{y}_{\ell^\infty} :&= \max_{\xi=1,\dots,N} |y_\xi|,  \\
\norm{y}_{\U{1}{p}} :&= \norm{y'}_{\lpe} \mbox{ for } p \in  [1,\infty].
\end{alignat*}
Correspondingly, we let $\lpe$ denote the space $\mathcal{U}$
equipped with the norm $\norm{\cdot}_{\lpe}$, and we let
$\U{1}{p}$ denote the space $\mathcal{U}$ equipped with the
norm $\norm{\cdot}_{\U{1}{p}}$.
Additionally, let $Y^*$ denote the topological dual of the Banach
space $Y$, and let $\U{-1}{p} := (\U{1}{q})^*$ where $p,q \in [1,\infty]$
with $\frac{1}{p}+\frac{1}{q}=1$.
We let $\A$ denote the atomistic region, $\C$ denote the
continuum region, and $\I$ denote the interface. We let
$\norm{\cdot}_{\lpe(\mathcal{P})}$ be the $\norm{\cdot}_{\lpe}$
norm taken over the set $\mathcal{P}$ for $\mathcal{P} = \A, \C, \mbox{ or } \I$.
We denote the closed ball of radius $r$ at $x$ in $X$ by
\begin{equation*}
B_X(x,r) := \{y \in X: \norm{y - x}_X \leq r\}
\end{equation*}
for $X$ one of the spaces $\lpe$ or $\U{1}{p}$.

We will write $\fv \Psi$ for the first variation of a differentiable energy functional $\Psi$.
The first variation is a map from $\Y$ into $\mathcal{U}^*$; we will let
$\fv \Psi(y)[u]$ denote the first variation of $\Psi$ at the deformation $y \in \Y$
evaluated on the test function $u \in \mathcal{U}$.
We will use the letter $u$ to denote a test function belonging to $\mathcal{U}$ throughout the remainder of the paper.
The letter $y$ will be used to denote a deformation belonging to $\Y$.
We warn the reader that in expressions such as $\fv \Psi(y)[u]$,
$u$ denotes an arbitrary test function, not the displacement corresponding to the deformation $y$.
Similarly, we will let $\fv^2 \Psi$ denote
the second variation of $\Psi$. The second variation can be interpreted either as a map from
$\Y$ into the space of bilinear forms on $\mathcal{U}$, or as a map
from $\Y$ into $L(\mathcal{U}, \mathcal{U}^*)$. We will let $\fv^2 \Psi(y)[u,v]$ denote the
second variation of $\Psi$ at $y \in \Y$ evaluated on the test functions
$u,v \in \mathcal{U}$.

We will need the following version of the Inverse Function Theorem
which appears as Lemma~1 in~\cite{ortner:qnl1d}.
\begin{theorem}\label{thm: ift}
\emph{(Inverse Function Theorem)} Let $X$ and $Y$ be Banach
spaces, let $A$ be an open subset of $X$, and let $\F :A
\rightarrow Y$ be a $C^1$ function. Let $x_0 \in A$ and
suppose:
\begin{enumerate}
\item $\norm{\F (x_0)}_Y \leq \eta$,
and $\bignorm{\left (\fv \F(x_0)\right )^{-1}}_{L(Y,X)} \leq \sigma;$
\item $\overline{B_X(x_0, 2 \eta \sigma)} \subset A;$
\item $\norm{\fv \F(x_1) - \fv \F(x_2)}_ {L(X,Y)} \leq L$ for
    all $x_1, x_2 \in X$ with $\norm{x_1 - x_2}_X \leq 2 \eta \sigma$;
\item
$2 L \sigma^2 \eta < 1$.
\end{enumerate}
Then there exists $x \in X$ so that $\F(x) = 0$ and $\norm{x - x_0}_X < 2 \eta \sigma$.
\end{theorem}
In Section~5, we will use Theorem~\ref{thm: ift} to
prove {\em a priori} existence with error estimates for the
BQCE and BQNL energies. We will apply Theorem~\ref{thm: ift}
using bounds on $\eta$ derived from the modeling estimates in
Section~\ref{subsec: modeling} and bounds on $\sigma$ derived
from the stability estimates given in Section~4.
\section{Modeling Error for the BQC method}\label{subsec: modeling}

In Theorem~\ref{thm: modeling}, we estimate the $\U{-1}{p}$ norms of
the modeling errors of the BQCE and BQNL energies.
\begin{theorem}[Modeling error in $\U{-1}{p}$] \label{thm: modeling}
Let $y \in \Y$ with $\min_{\xi \in \mathbb{Z}}
  y'_\xi >0$.
\begin{enumerate}
\item
 Let $\bqce$ be the BQCE energy with blending function $\gamma$.
  Recall from equation~\eqref{eqn: alpha and beta for BQCE} that $\bqce$ is the BQC energy with blending functions
  $\alpha_\xi := \overline{\gamma_\xi}$,
  and $\beta_\xi := 1 - \frac{\gamma_{\xi+1} + \gamma_{\xi-1}}{2}$. We have
 \begin{equation} \label{eqn: modeling for bqce}
 \begin{split}
   \umpnorm{\fv\nl(y) - \fv\bqce(y)} &\leq \bar{C}_1 \norm{\Delta^2\alpha_\xi}_{\lpe} + \eps \Ctwo\norm{\Delta \beta_\xi y''_\xi}_{\lpe} \\
   &+ \eps^2 \{\Ctwo\norm{(1-\beta_{\xi-1})y'''_\xi}_{\lpe} +\Cthr\norm{(1-\beta_{\xi})(y''_\xi)^2}_{\lpe} \}
 \end{split}
 \end{equation}
where $\bar{C}_i:= C_i(r^{min})$,
$i = 1,2,3,$ for $r^{min}:=2\,\min_{\xi \in \mathbb{Z}} y'_\xi.$
\item
Let $\bqnl$ be the BQNL energy with blending function $\beta$. We have
 \begin{equation} \label{eqn: modeling for bqnl}
 \begin{split}
   \umpnorm{\fv\nl(y) - \fv\bqnl(y)} &\leq \eps \Ctwo\norm{\Delta \beta_\xi y''_\xi}_{\lpe} \\
   &+ \eps^2 \{\Ctwo\norm{(1-\beta_{\xi-1})y'''_\xi}_{\lpe} +\Cthr\norm{(1-\beta_{\xi})(y''_\xi)^2}_{\lpe} \}
 \end{split}
 \end{equation}
where $\bar{C}_i:= C_i(r^{min})$,
$i = 2,3,$ for $r^{min}:=2\,\min_{\xi \in \mathbb{Z}} y'_\xi.$
\end{enumerate}
\end{theorem}

\begin{proof}
We begin by writing down the first variations of $\bqcab$ and
$\nl$. For $y \in \Y$ and $u \in \U{1}{p},$ we have
\begin{equation*}
\begin{split}
\fv\bqcab(y)[u] &= \eps \periodsum\p'(y'_\xi)u'_\xi +2\axi\p'(2y'_\xi)u'_\xi +\beta_\xi \p'(y'_\xi + y'_{\xi + 1})(u'_\xi + u'_{\xi + 1})\\
              &= \eps \periodsum \{ \p'(y'_\xi) +2\axi\p'(2y'_\xi) +\beta_{\xi-1}\p'(y'_{\xi-1}+y'_{\xi}) +\beta_\xi\p'(y'_\xi+y'_{\xi+1}) \}u'_\xi.
\end{split}
\end{equation*}
The first variation of $\nl$ is a special case of the above. We have
\begin{equation*}
\fv\nl(y)[u] = \eps \periodsum \{ \p'(y'_\xi) +\p'(y'_{\xi-1} + y'_{\xi})+\p'(y'_\xi + y'_{\xi + 1}) \}u'_\xi.
\end{equation*}
Using these formulas we compute the {\em modeling error}
\begin{align}\label{eqn: first truncation formula}
M[u] :&= \fv\nl(y)[u]-\fv\bqcab(y)[u]\\
     & = \eps\periodsum \left [ \{(1-\beta_{\xi-1})\p'(y'_{\xi-1}+y'_{\xi})-\axi\p'(2y'_{\xi})\}
                          +  \{(1-\beta_{\xi})\p'(y'_{\xi} + y'_{\xi+1})-\axi \p'(2y'_{\xi})\} \right ] u'_\xi.\notag
\end{align}
Now we expand the terms
\begin{equation*}
G_\xi := \{(1-\beta_{\xi-1})\p'(y'_{\xi-1}+y'_{\xi})-\axi \p'(2y'_{\xi})\} \quad\text{and}\quad
H_\xi := \{ (1-\beta_{\xi})\p'(y'_{\xi} + y'_{\xi+1})-\axi \p'(2y'_{\xi}) \}
\end{equation*}
in~\eqref{eqn: first truncation formula} at $2y'_\xi$ using Taylor's
theorem. We obtain
\begin{align}
G_\xi &= (1-\beta_{\xi-1}- \axi)\p'(2y'_{\xi})
    +(1-\beta_{\xi-1})(-\eps\p''(2y'_{\xi})y''_{\xi-1}+\frac{\eps^2}{2}\p'''(\oc_{1,\xi})(y''_{\xi-1})^2) \mbox{\quad and}\label{eqn: modeling A}\\
H_\xi &= (1-\beta_{\xi}-\axi)\p'(2 y'_{\xi})
    +(1-\beta_{\xi})(\eps\p''(2y'_{\xi})y''_\xi+\frac{\eps^2}{2}\p'''(\oc_{2,\xi})(y''_\xi)^2), \label{eqn: modeling B}
\end{align}
where $\oc_{1,\xi} \in \mbox{conv} \{2y'_\xi, y'_{\xi-1} +
y'_\xi\}$, and $\oc_{2,\xi} \in \mbox{conv}\{2y'_\xi,
y'_{\xi+1} + y'_\xi\}$. When we substitute~\eqref{eqn:
modeling A} and~\eqref{eqn: modeling B} into~\eqref{eqn:
first truncation formula}, we obtain
\begin{equation}\label{eqn: second truncation formula}
\begin{split}
M[u] &=\eps \periodsum 2 \{ 1 - \axi-\bbar_\xi \} \p'(2y'_\xi) u'_\xi \\
     & \quad +\eps \{ (1-\beta_\xi)\p''(2y'_\xi)y''_\xi -(1-\beta_{\xi-1})\p''(2y'_\xi)y''_{\xi-1} \} u'_\xi \\
     & \quad +\frac{\eps^2}{2} \{ (1-\beta_\xi)\p'''(\oc_{2,\xi})(y''_\xi)^2 +(1-\beta_{\xi-1})\p'''(\oc_{1,\xi})(y''_{\xi-1})^2 \} u'_\xi.
\end{split}
\end{equation}
Expanding the term of order $\eps$ on the right hand side of~\eqref{eqn: second truncation formula},
we compute
\begin{equation}\label{eqn: general truncation formula}
\begin{split}
 M[u] =  \eps \periodsum 2 &\{ 1 - \axi-\bbar_\xi \} \p'(2y'_\xi) u'_\xi  \\
 &+ \{-\eps \Delta\beta_\xi\p''(2y'_{\xi})y''_\xi + \eps^2 (1-\beta_{\xi-1}) \p''(2y'_{\xi})y'''_{\xi} \} u'_\xi \\
&+\frac{\eps^2}{2}\{(1-\beta_\xi)\p'''(\oc_{2,\xi})(y''_\xi)^2+(1-\beta_{\xi-1})\p'''(\oc_{1,\xi})(y''_{\xi-1})^2\}u'_\xi.
\end{split}
\end{equation}

We recall from equation~\eqref{eqn: alpha and beta for BQNL} that the BQNL energy $\bqnl$ with blending function $\beta$
is the BQC energy $\bqc$ with $\alpha_\xi = 1 - \overline{\beta_\xi}$.
In that case, we compute that for the BQNL energy $\bqnl$, the term of order zero in equation~\eqref{eqn: general truncation formula} vanishes,
and so
\begin{equation*}
\begin{split}
 \fv\nl(y)[u] - \fv\bqnl(y)[u] =  \eps \periodsum \{ - &\eps \Delta\beta_\xi\p''(2y'_{\xi})y''_\xi
 + \eps^2 (1-\beta_{\xi-1}) \p''(2y'_{\xi})y'''_{\xi} \} u'_\xi \\
+&\frac{\eps^2}{2}\{(1-\beta_\xi)\p'''(\oc_{2,\xi})(y''_\xi)^2+(1-\beta_{\xi-1})\p'''(\oc_{1,\xi})(y''_{\xi-1})^2\}u'_\xi.
\end{split}
\end{equation*}
Therefore, by H\"{o}lder's inequality,
\begin{equation*}
 |\fv\nl(y)[u] - \fv\bqnl(y)[u]| \leq \left \{ \eps \Ctwo\norm{\Delta \beta_\xi y''_\xi}_{\lpe} +
\eps^2 \Ctwo\norm{(1-\beta_{\xi-1})y'''_\xi}_{\lpe} +
\eps^2 \Cthr\norm{(1-\beta_{\xi})(y''_\xi)^2}_{\lpe}  \right \} \norm{u'}_{\lqe}
\end{equation*}
where $q := \frac{p}{p-1}$.
Thus,
\begin{equation*}
\umpnorm{\fv\nl(y)[u] - \fv\bqnl(y)[u]} \leq \eps \Ctwo\norm{\Delta \beta_\xi y''_\xi}_{\lpe} \\
   + \eps^2 \{\Ctwo\norm{(1-\beta_{\xi-1})y'''_\xi}_{\lpe} +\Cthr\norm{(1-\beta_{\xi})(y''_\xi)^2}_{\lpe} \}.
\end{equation*}
This proves the first claim made in the statement of the theorem.

On the other hand, for the BQCE energy $\bqce$ there is a ghost force
arising from the term of order zero on the right hand side of equation~\eqref{eqn: general truncation formula}.
Using formulas~\eqref{eqn: alpha and beta for BQCE}, we compute
\begin{equation}\label{eqn: ghost force}
2 \{1 - \axi-\bbar_\xi\} = \alpha_{\xi+1}-2\axi + \alpha_{\xi-1} = \Delta^2\alpha_\xi.
\end{equation}
Substituting this expression in equation~\eqref{eqn: general truncation formula}, we derive
\begin{equation*}
\begin{split}
 \fv\nl(y)[u] - \fv\bqce(y)[u] =  \eps \periodsum &\Delta^2\alpha_\xi \p'(2y'_\xi) u'_\xi \\
 &+\{ - \eps \Delta\beta_\xi\p''(2y'_{\xi})y''_\xi
 + \eps^2 (1-\beta_{\xi-1}) \p''(2y'_{\xi})y'''_{\xi} \} u'_\xi \\
&+\frac{\eps^2}{2}\{(1-\beta_\xi)\p'''(\oc_{2,\xi})(y''_\xi)^2+(1-\beta_{\xi-1})\p'''(\oc_{1,\xi})(y''_{\xi-1})^2\}u'_\xi.
\end{split}
\end{equation*}
Thus,
\begin{equation*}
\begin{split}
 |\fv\nl(y)[u] - \fv\bqce(y)[u]| &\leq \bigg\{ \bar{C}_1 \norm{\Delta^2\alpha_\xi}_{\lpe} + \eps \Ctwo\norm{\Delta \beta_\xi y''_\xi}_{\lpe} \\
   &+ \eps^2  \Ctwo\norm{(1-\beta_{\xi-1})y'''_\xi}_{\lpe}
   + \eps^2 \Cthr\norm{(1-\beta_{\xi})(y''_\xi)^2}_{\lpe}  \bigg\} \norm{u'}_{\lqe}.
\end{split}
\end{equation*}
This proves the first claim made in the statement of the theorem.
\end{proof}

\begin{remark}[BQNL is patch test consistent and BQCE is not patch test consistent]
Estimate~\eqref{eqn: modeling for bqnl} implies that BQNL is patch test consistent.
For observe that if $y^F$ is the uniform deformation,
then $(y^F)''_\xi = (y^F)'''_\xi = 0$ for all $\xi \in \mathbb{Z}$.
Thus, by~\eqref{eqn: modeling for bqnl},
$\norm{\fv \nl(y^F) - \fv \bqnl(y^F)}_{\U{-1}{p}} = \norm{\fv \bqnl(y^F)}_{\U{-1}{p}} = 0$.
On the other hand, observe that the term $\bar{C}_1 \norm{\Delta^2\alpha_\xi}_{\lpe}$
which appears on the right hand side of estimate~\eqref{eqn: modeling for bqce}
does not vanish under uniform strain.
This reflects the fact that BQCE is not patch test consistent.
\end{remark}

\begin{remark}[Interpretation of modeling estimates]\label{rem: clarification of modeling}
We will now give a detailed interpretation of the each term in estimates~\eqref{eqn: modeling for bqnl} and~\eqref{eqn: modeling for bqce}.
First, we consider the term
\begin{equation*}
\eps^2 \{\Ctwo\norm{(1-\beta_{\xi-1})y'''_\xi}_{\lpe} +\Cthr\norm{(1-\beta_{\xi})(y''_\xi)^2}_{\lpe} \}.
\end{equation*}
Observe that the function $1-\beta_\xi$ is supported in the interface and continuum, and that it is identically
one throughout the continuum. This suggests that the term arises from the error of the Cauchy-Born model.
Recall that the Cauchy-Born energy is the
BQNL energy with blending function $\beta_\xi = 0$.
Therefore, by the modeling estimate for BQNL given in
Theorem~\ref{thm: modeling} we have
\begin{equation*}
\norm{\fv\nl(y) - \fv\Phi^{cb}(y)}_{\U{-1}{p}} \leq \eps^2 \{\Ctwo\norm{y'''_\xi}_{\lpe} +\Cthr\norm{(y''_\xi)^2}_{\lpe}\}.
\end{equation*}
Consequently, we will call the term $\eps^2 \{\Ctwo\norm{(1-\beta_{\xi-1})y'''_\xi}_{\lpe} +\Cthr\norm{(1-\beta_{\xi})(y''_\xi)^2}_{\lpe} \}$
the \emph{Cauchy-Born error}.

Next, we consider the term $\eps \Ctwo\norm{\Delta \beta_\xi y''_\xi}_{\lpe}$.
 We observe that $\Delta\beta_\xi$ is supported
in the interface. Thus, the term $\eps \Ctwo\norm{\Delta \beta_\xi y''_\xi}_{\lpe}$
arises from the error caused by coupling the atomistic and
continuum models in the interface. We will call this term the
\emph{coupling error}.

Finally, we consider the term $\Cone\norm{\Delta^2\alpha_\xi}_\lpe$ which appears
in the modeling estimate for BQCE but not in the estimate for BQNL.
Let $y^F$ be the uniform deformation $y^F_\xi := F\eps \xi$.
We call $\fv\bqce(y^F)$ the \emph{ghost force} associated
with the energy $\bqce$, and we observe that under uniform strain
formula~\eqref{eqn: general truncation formula} reduces to
\begin{equation*}
\fv\bqce(y^F)[u] =  \eps \periodsum -\Delta^2\alpha_\xi \p'(2F) u'_\xi.
\end{equation*}
Thus, we see that $\p'(2F)\norm{\Delta^2\alpha_\xi}_\lpe$ is the $\U{-1}{p}$
norm of the ghost force. We will call $\Cone\norm{\Delta^2\alpha_\xi}_\lpe$
the \emph{ghost force error}.

\end{remark}

\begin{remark}[Dependence of ghost force error on interface size]\label{rem: ghost force}
We will now analyze the dependence of the ghost force error on the size of the interface and the shape of the blending function.
Our analysis will lead to an estimate for the rate at which the ghost force error decreases with the number of atoms in the interface,
and to an optimal family of blending functions for the BQCE method.
Consider the BQCE energy $\bqce$ with blending function $\gamma$ as depicted in Figure~\ref{fig:blendingfigure}.

\begin{figure}
  \begin{center}
    \includegraphics[width=11cm]{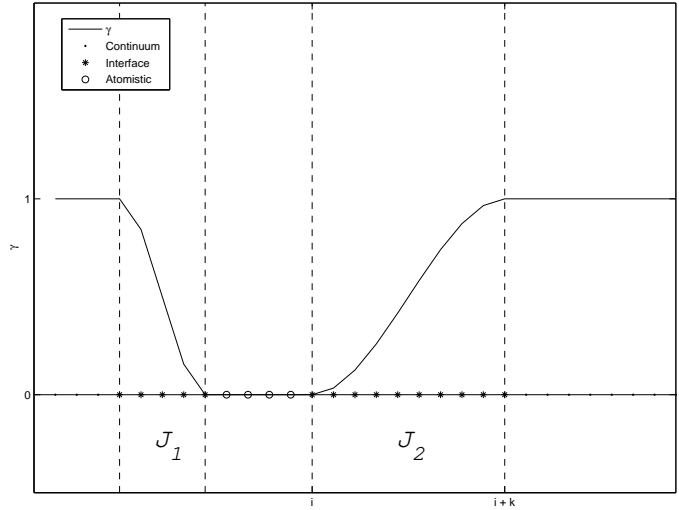}
    \caption{\label{fig:blendingfigure} Graph of the blending function $\gamma$
    with interface region $\mathcal{I} = \J_1 \cup \J_2$.
      }
  \end{center}
\end{figure}

Each transition between the atomistic and continuum models
results in some ghost force. We will consider the ghost force
which arises from a single transition from the atomistic model
to the continuum model. The total ghost force is, of course,
the sum of the ghost forces due to each transition. Let $\J_2
:= \{i, \dots, i + \kk\}$ be part of the interface, which we
will denote simply as $\J$ is the following.  Assume that
atoms $i-2, i-1,$ and $i$ are in the atomistic region, and that
atoms $i+\kk$ and $i+\kk+1$ are in the continuum. That is,
suppose that $\gamma(\xi) = 0$ for $\xi \in \{i-2, \dots, i
\}$, and that $\gamma(\xi) = 1$ for $\xi \in \{i + \kk, i + \kk
+ 1\}$. We will estimate the ghost force due to the transition
which occurs over region $\J$.

 First, we construct a  blending function which implements such a transition.
 Let $\gs: [0,1] \rightarrow [0,1]$ be a twice continuously
 differentiable function such that $\gs(0) = 0$, $\gs(1) =1$, and
 $\gs'(0) = \gs'(1) = 0$. Extend $\gs$ to a function defined on $\Real$
 by taking $\gs(x) = 0$ for $x \in (-\infty,0)$ and $\gs(x) = 1$ for $x \in (1,\infty)$.
 Now let $\J^+ := \{i-2, \dots, i+\kk+1\}$, and define $\gamma_\J :\J^+ \rightarrow [0,1]$ by
 $\gamma_\J(\xi) = \gs\left(\frac{\xi -i}{\kk} \right)$.
 We call $\gs$ the \emph{shape} of the blending function $\gamma_\J$.

 We will show
 \begin{equation*}
 \norm{\Delta^2\gamma_\J}_{\lpe(\J)} \leq \eps^{\frac{1}{p}} \kk^{\frac{1}{p}-2} \bignorm{\frac{d^2\gs}{dx^2}}_{L^p([0,1])}
 \end{equation*}
 for all $p \in [1,\infty]$.
 Suppose $p \in [1, \infty)$, and compute
\begin{equation}\label{eqn: quadrature sum}
 \norm{\Delta^2\gamma_\J}_{\lpe(\J)} = \eps^2 \norm{\gamma''_\J}_{\lpe(\J)}
 = \eps^{\frac{1}{p}} \kk^{\frac{1}{p}-2}  \left \{\frac{1}{\kk} \sum_{\xi = 0}^{\kk}
\bigabs{\frac{\gs \left(\frac{\xi+1}{\kk}\right) - 2\gs \left (\frac{\xi}{\kk} \right ) + \gs \left( \frac{\xi-1}{\kk} \right)}{\frac{1}{k^2}}
}^p \right \}^{\frac{1}{p}}.
\end{equation}
By Taylor's theorem,
\begin{equation}\label{eqn: expansion of second difference}
 \frac{\gs \left(\frac{\xi+1}{\kk}\right) - 2\gs \left (\frac{\xi}{\kk} \right ) + \gs \left( \frac{\xi-1}{\kk} \right)}{\frac{1}{k^2}}
 = \int_{\frac{\xi -1}{\kk}}^{\frac{\xi + 1}{\kk}} \frac{d^2\gs}{dx^2} (s) \tk_\xi(s) ds,
\end{equation}
where
\begin{equation*}
 \tk_\xi(s) :=
\begin{cases}
 -\kk^2 \bigabs{s -  \frac{\xi}{\kk}} + \kk & \mbox{if } s \in \left [\frac{\xi-1}{\kk},\frac{\xi+1}{\kk} \right], \\
 0 & \mbox{otherwise}.
\end{cases}
\end{equation*}
We note that~\eqref{eqn: expansion of second difference} holds for $\xi=i$ and $\xi = i+\kk$ only since we have
assumed that $\gs'(0) = \gs'(1) =0$.
Now observe that $\tk_\xi$ is a non-negative function whose mass is one, and that $x \mapsto |x|^p$
is convex for $p \geq 1$. Therefore, Jensen's inequality implies
\begin{equation*}
 \bigabs{\int_{\frac{\xi -1}{\kk}}^{\frac{\xi + 1}{\kk}} \frac{d^2\gs}{dx^2} (s) \tk_\xi(s) \, ds}^p
\leq \int_{\frac{\xi -1}{\kk}}^{\frac{\xi + 1}{\kk}} \bigabs{\frac{d^2\gs}{dx^2} (s)}^p \tk_\xi(s) \, ds.
\end{equation*}
This yields the estimate
\begin{align*}
\norm{\Delta^2\gamma_\J}_{\lpe(\J)} &\leq \eps^{\frac{1}{p}} \kk^{\frac{1}{p}-2}  \left \{ \sum_{\xi = 0}^{\kk}
       \int_{\frac{\xi -1}{\kk}}^{\frac{\xi + 1}{\kk}} \bigabs{\frac{d^2\gs}{dx^2} (s)}^p \frac{\tk_\xi(s)}{\kk} \, ds  \right \}^{\frac{1}{p}}
\\
&= \eps^{\frac{1}{p}} \kk^{\frac{1}{p}-2} \left \{
       \int_{0}^{1} \bigabs{\frac{d^2\gs}{dx^2} (s)}^p \sum_{\xi = 0}^{\kk} \frac{\tk_\xi(s)}{\kk} \, ds \right \}^{\frac{1}{p}}
= \eps^{\frac{1}{p}} \kk^{\frac{1}{p}-2} \bignorm{\frac{d^2\gs}{dx^2}}_{L^p([0,1])}.
\end{align*}
for all $p \in [1, \infty)$.
By a similar argument, one can show
\begin{equation*}
 \norm{\Delta^2\gamma_\J}_{\ell^\infty(\J)} \leq \kk^{-2} \bignorm{\frac{d^2\gs}{dx^2}}_{L^\infty([0,1])}.
\end{equation*}

Recall from Remark~\ref{rem: clarification of modeling}
that the ghost force error due to the transition over $\J$ is
\begin{equation*}
 \norm{\Delta^2\alpha_\J}_{\lpe(\J)} := \norm{\Delta^2 \bar{\gamma}_\J }_{\lpe(\J)}.
\end{equation*}
By Minkowski's inequality and the estimates above, we have
\begin{equation}
 \norm{\Delta^2\alpha_\J}_{\lpe(\J)} = \norm{\Delta^2 \bar{\gamma}_\J}_{\lpe(\J)}
                                          \leq \norm{\Delta^2\gamma_\J}_{\lpe(\J)}
					\leq  \eps^{\frac{1}{p}} \kk^{\frac{1}{p}-2} \bignorm{\frac{d^2\gs}{dx^2}}_{L^p([0,1])} \label{eqn: ghost force estimate},
\end{equation}
for all $p\in[1,\infty]$.
This gives the dependence of the ghost force error on $\eps$, the blending function,
and $\kk$. We remark that inequality~\eqref{eqn: ghost force estimate} is sharp.
In fact, if $\gs$ is sufficiently smooth, then $\norm{\Delta^2\alpha_\J}_{\lpe(\J)}$
converges to $\eps^{\frac{1}{p}} \kk^{\frac{1}{p}-2}\bignorm{\frac{d^2\gs}{dx^2}}_{L^p([0,1])}$ as
$\kk$ tends to infinity.

We pause here to explain why we assumed $\gs'(0) = \gs'(1) = 0$.
It can be shown that no estimate of order $\eps^{\frac{1}{p}} \kk^{\frac{1}{p}-2}$ holds unless $\gs'(0) = \gs'(1) = 0$.
We leave the proof of this fact to the reader, and will instead give an illustrative example.
Suppose that $\gs$ were the linear polynomial with $\gs(0) = 0$ and $\gs(1) = 1$, so $\gs$ is not differentiable at $0$ or $1$.
Then
\begin{equation}\label{eqn: ghost force scaling for linear blending}
\norm{\Delta^2\gamma_\J}_{\lpe(\J)} = 2^{\frac{1}{p}} \eps^{\frac{1}{p}} \kk^{-1}.
\end{equation}
In Section~\ref{subsec: a priori existence}, we use an estimate
of the $\U{-1}{2}$ norm of the modeling error in order to
obtain convergence results. Thus, we are particularly
interested in the $\U{-1}{2}$ norm of the ghost force error. By
estimate~\eqref{eqn: ghost force estimate}, if $\gs$ satisfies
$\gs'(0) = \gs'(1) = 0$ then the $\U{-1}{2}$ norm of the ghost
force error decreases with $\kk$ as $\kk^{-\frac{3}{2}}$. On
the other hand, if $\gs$ is the linear blending function then
we see that the ghost force error decreases as $\kk^{-1}$. We
conclude that if a large blending region is desired, it is best
to choose a blending function which satisfies $\gs'(0) =
\gs'(1) = 0$ so that the faster rate of decrease is obtained.
In particular, we suspect that the linear polynomial $\gs$ is a
poor choice for the shape of the BQCE blending function.

We will now use estimate~\eqref{eqn: ghost force estimate} to
derive an optimal shape for the blending function of the BQCE
energy. As discussed above, we will use the $\U{-1}{2}$
modeling estimate to prove our error results in
Section~\ref{subsec: a priori existence}. Thus, we would like
to find a family of blending functions which minimizes the
$\U{-1}{2}$ ghost force. Estimate~\eqref{eqn: ghost force
estimate} suggests that we should choose the shape $\gs$ which
minimizes $\bignorm{\frac{d^2 \gs}{dx^2}}_{L^2([0,1])}$ subject
to $\gs(0) = 0$, $\gs(1) = 1$, and $\gs'(0) = \gs'(1) = 0$. The
Euler-Lagrange equation for this minimization problem is
\begin{equation*}
 \frac{d^4 \gs}{dx^4}(x) = 0 \mbox{ for all } x \in (0,1).
\end{equation*}
Thus, the optimal shape $\gs$ is the cubic polynomial which
satisfies the constraints $\gs(0) = 0$, $\gs(1) = 1$,
 and $\gs'(0) = \gs'(1) = 0$.
\end{remark}
\begin{remark}[Dependence of coupling error on interface size]\label{rem: coupling error}
Now we examine the dependence of the coupling error
on the number of atoms in the interface and the shape of the blending function.
First, we consider the coupling error of the BQNL energy.
Following the notation established in Remark~\ref{rem: ghost force},
let $\J := \{i, \dots, i + \kk\}$ be contained in the interface.
Let $\beta_s: [0,1] \rightarrow [0,1]$ be a continuously differentiable function
with $\beta_s(0) = 0$ and $\beta_s(1) =1$. Extend $\beta_s$ to a function defined on $\Real$
by taking $\beta_s(x) = 0$ for $x \in (-\infty,0)$ and $\beta_s(x) = 1$ for $x \in (1,\infty)$.
Now let $\beta_\J: \{i, \dots, i + \kk\} \rightarrow [0,1]$ be defined by
$\beta_\J(\xi) = \beta_s(\frac{\xi- i}{\kk})$.
Using an argument similar to the one given in Remark~\ref{rem: ghost force},
one can show
\begin{equation*}
 \norm{\Delta \beta_\J}_{\lpe(\J)} \leq C_\beta \eps^{\frac{1}{p}} \kk^{\frac{1}{p}-1}
\end{equation*}
for all $p \in [1, \infty]$.
The result above yields the estimate
\begin{equation*}
\eps \norm{\Delta\beta_\J y''}_{\lpe(\J)} \leq \eps \norm{\Delta\beta_\J}_{\lpe(\J)} \norm{y''}_{\ell^\infty(\J)}
                                               \leq \eps^{1 + \frac{1}{p}} \kk^{\frac{1}{p}-1} C_\beta \norm{y''}_{\ell^\infty(\J)}.
\end{equation*}
This gives the dependence of the coupling error of the BQNL energy
on $\eps$, the blending function, and the size of the interface.

We now address the coupling error of the BQCE method.
Let the blending function, $\gamma_\J$, of the BQCE energy be defined as in Remark~\ref{rem: ghost force}.
We recall from equation~\eqref{eqn: alpha and beta for BQCE} that the BQCE energy with blending function $\gamma_\J$
is the BQC energy, $\Phi_{\alpha, \beta}$, with
\begin{equation*}
\beta_\xi = 1 - \frac{\gamma_{\xi+1} + \gamma_{\xi-1}}{2}.
\end{equation*}
Thus, using Minkowski's inequality and the result for the BQNL energy,
we make the estimate
\begin{equation}\label{eqn: scaled coupling error estimate}
\eps \norm{\Delta\beta_\J y''}_{\lpe(\J)} :=
 \eps \bignorm{\Delta \left (1 - \frac{\gamma_{\xi+1} + \gamma_{\xi-1}}{2} \right ) y''}_{\lpe(\J)}
  \leq \eps \norm{\Delta\gamma_\J y''}_{\lpe(\J)} \leq \eps^{1 + \frac{1}{p}} \kk^{\frac{1}{p}-1} C_\gamma \norm{y''}_{\ell^\infty(\J)}.
\end{equation}
This gives the dependence of the coupling error of the BQCE method
on $\eps$, the blending function, and the size of the interface.
\end{remark}
\begin{remark}[Higher order estimates]\label{higher}
Suppose we let the measure of the interface be a fixed fraction $\nu$ of the total
measure of the domain as $\eps$ tends to zero. Then the number of atoms $\kk$
in the interface would be approximately $\frac{\nu}{\eps}$,
 and estimates~\eqref{eqn: ghost force estimate}
and~\eqref{eqn: scaled coupling error estimate} would reduce to
\begin{equation*}
 \norm{\Delta^2\alpha_\J}_{\lpe(\J)} \leq C \eps^2 \nu^{\frac{1}{p}-2} \quad\text{and}\quad
 \eps \norm{\Delta\beta_\J y''}_{\lpe(\J)} \leq D \eps^2 \nu^{\frac{1}{p}-1}.
\end{equation*}
\end{remark}

\begin{remark}[Newton's Third Law]
We observe that the forces arising on each atom due to the BQC energy
can be decomposed into a sum of central forces which satisfy Newton's Third Law.
In particular, the forces arising from both the BQCE and BQNL energies satisfy Newton's Third Law.
\end{remark}

\section{Stability of the BQC method}\label{subsec: stability}
First we derive expressions for the second variations of the atomistic
and BQC energies.  We have
\begin{align}\label{eqn: formula for Hessian}
\fv^2 \bqcab(y)[u,u] &= \eps \periodsum \p''(y'_{\xi}) |u'_\xi|^2 + 4 \alpha_\xi \p''(2y'_\xi)|u'_\xi|^2
+ \beta_\xi \p''(y'_{\xi} + y'_{\xi+1})|u'_\xi + u'_{\xi+1}|^2 \\
 &= \eps \periodsum \p''(y'_{\xi})|u'_\xi|^2 +4\alpha_\xi \p''(2y'_\xi)|u'_\xi|^2 +\beta_\xi \p''(y'_{\xi}
 + y'_{\xi + 1}) [2 |u'_\xi|^2 + 2 |u'_{\xi + 1}|^2 - \eps^2 |u''_\xi|^2] \notag\\
 &= \eps \periodsum \abar_\xi |u'_\xi|^2 + \eps^2 \bar{B}_\xi |u''_\xi|^2,\notag
\end{align}
where
\begin{equation}\label{eqn: bqc hessian cfcts}
\begin{split}
\abar_\xi &:=  \p''(y'_{\xi}) + 4 \left [ \half \beta_\xi \p''(y'_{\xi} + y'_{\xi+1})
                              + \half \beta_{\xi-1} \p''(y'_{\xi-1} + y'_{\xi})
                              + \alpha_\xi \p''(2y'_\xi) \right ] \mbox{ and} \\
\bar{B}_\xi &:= - \beta_\xi \p''(y'_{\xi} + y'_{\xi+1}).
\end{split}
\end{equation}
The second variation of the atomistic energy is, of course, a special
case of the above. We have
\begin{align*}
\fv^2 \nl(y)[u,u] &= \eps \periodsum \p''(y'_{\xi}) + 4 \left \{\half \p''(y'_{\xi} + y'_{\xi+1}) +
\half \p''(y'_{\xi} + y_{\xi -1})] \right \} |u'_\xi|^2 +(-\eps^2 \p''(y'_{\xi} + y'_{\xi+1}) |u''_\xi|^2) \\
                &= \eps \periodsum A_\xi |u'_\xi|^2 + \eps^2 B_\xi |u''_\xi|^2,
\end{align*}
where
\begin{equation}\label{eqn: nl hessian cfcts}
\begin{split}
A_\xi &:=   \p''(y'_{\xi}) + 4 \left [ \half \p''(y'_{\xi} + y'_{\xi+1})
                              + \half  \p''(y'_{\xi-1} + y'_{\xi}) \right ] \mbox{ and} \\
B_\xi &:= - \p''(y'_{\xi} + y'_{\xi+1}).
\end{split}
\end{equation}

We begin our analysis with a lemma bounding $|A_\xi - \abar_\xi|$.
In Remark~\ref{rem: clarification of modeling}, we interpreted
each term which appears in the estimate below, and in Remarks~\ref{rem: ghost force}
and~\ref{rem: coupling error}
we explained how each term depends on $\eps$, the blending
function, and the number of atoms in the interface. We concluded
that
\begin{equation*}
\eps \norm{\Delta \beta_\xi y''_\xi}_{\ell^\infty} \lesssim \eps \kk^{-1} \norm{y''}_{\ell^{\infty}(\I)}\quad\text{and}\quad
\norm{\Delta^2\alpha_\xi}_{\ell^\infty}  \lesssim \kk^{-2},
\end{equation*}
where $\kk$ is the number of atoms in the interface.
The reader should keep these scalings in mind throughout
Section~\ref{subsec: stability}.

\begin{lemma}\label{lem: difference of hessian cfcts}
\begin{verbatim} \end{verbatim}
\begin{enumerate}
\item Let $\bqce$ be the BQCE energy with blending function $\gamma$.  We have
\begin{equation*}
\begin{split}
\max_\xi |\abar_\xi -A_\xi |\leq &2\Ctwo\norm{\Delta^2\alpha_\xi}_{\ell^\infty}  \\
+&\stabbound
\end{split}
\end{equation*}
where $\alpha_\xi := \overline{\gamma_\xi}$,
$\beta_\xi := 1 - \frac{\gamma_{\xi+1} + \gamma_{\xi-1}}{2},$
and $\bar{C}_i = C_i(r^{min})$,
$i = 2,3,4,$ for $r^{min}:=2\,\min_{\xi \in \mathbb{Z}} y'_\xi.$
\item Let $\bqnl$ be the BQNL energy with blending function $\beta$.  We have
\begin{equation*}
\max_\xi |\abar_\xi -A_\xi | \leq \stabbound
\end{equation*}
where $\bar{C}_i = C_i(r^{min})$,
$i = 3,4,$ for $r^{min}:=2\,\min_{\xi \in \mathbb{Z}} y'_\xi.$
\end{enumerate}
\end{lemma}
\begin{proof}
The proof of the lemma is extremely similar to the proof of our modeling estimate above.
For the general BQC energy $\bqcab$ we compute
\begin{equation*}
  \abar_\xi - A_\xi = 2(\beta_\xi -1) \p''(y'_{\xi} + y'_{\xi+1}) + 2(\beta_{\xi-1} -1) \p''(y'_{\xi-1} + y'_{\xi})  + 4\alpha_\xi \p''(2y'_\xi)
\end{equation*}
using formulas~\eqref{eqn: bqc hessian cfcts} and~\eqref{eqn:
nl hessian cfcts}. Then we expand all terms above at $2y'_\xi$
using Taylor's theorem. So,
\begin{equation}\label{eqn: second formula for diff of hess cfcts}
\begin{split}
\abar_\xi - A_\xi &=  4 \{\bbar_\xi + \alpha_\xi - 1\} \p''(2y'_\xi) \\
&+2\eps\p'''(2y'_\xi) \{(\beta_\xi - 1)y''_\xi -(\beta_{\xi-1}-1)y''_{\xi-1}\} \\ &+\eps^2\{(\beta_{\xi}-1)\p^{(4)}(\oc_{2,\xi})(y''_\xi)^2+(\beta_{\xi-1}-1)\p^{(4)}(\oc_{1,\xi})(y''_{\xi-1})^2\}
\end{split}
\end{equation}
where $\oc_{1,\xi} \in \mbox{conv} \{2y'_\xi, y'_{\xi-1} +
y'_\xi\}$, and $\oc_{2,\xi} \in \mbox{conv}\{2y'_\xi,
y'_{\xi+1} + y'_\xi\}$.
Now we expand the second term in curly braces on the right hand side of~\eqref{eqn: second formula for diff of hess cfcts}.
We obtain
\begin{equation} \label{eqn: third formula for diff of hess cfcts}
\begin{split}
\abar_\xi - A_\xi &= 4 \{\bbar_\xi + \alpha_\xi - 1\} \p''(2y'_\xi) \\
&+2\{\eps \Delta\beta_\xi y''_\xi +  \eps^2 (\beta_{\xi-1}-1)y'''_\xi\} \p'''(2y'_\xi) \\ &+\eps^2\{(\beta_{\xi}-1)\p^{(4)}(\oc_{2,\xi})(y''_\xi)^2+(\beta_{\xi-1}-1)\p^{(4)}(\oc_{1,\xi})(y''_{\xi-1})^2\}.
\end{split}
\end{equation}

We now consider the case of the BQCE energy.
As a consequence of equation~\eqref{eqn: ghost force}, we have
\begin{equation}\label{eqn: bqce hessian gf term}
4 \{\bbar_\xi + \alpha_\xi - 1\} = -2\Delta^2\alpha_\xi.
\end{equation}
Thus, by substituting~\eqref{eqn: bqce hessian gf term} into~\eqref{eqn: third formula for diff of hess cfcts},
we see that for the BQCE energy
\begin{equation*}
|\abar_\xi - A_\xi| \leq 2 \Ctwo \norm{\Delta^2\alpha_\xi}_{\ell^\infty} + \stabbound.
\end{equation*}
This proves the first claim made in the statement of the lemma.

On the other hand, we recall from the proof of Theorem~\ref{thm: modeling} that for the BQNL energy,
\begin{equation*}
 \bbar_\xi + \alpha_\xi - 1 = 0 \mbox{ for all } \xi \in \mathbb{Z}.
\end{equation*}
Therefore, for the BQNL energy,
the first term in curly braces in equation~\eqref{eqn: third formula for diff of hess cfcts} vanishes, and we have
\begin{equation*}
|\abar_\xi - A_\xi| \leq \stabbound.
\end{equation*}
This proves the second claim made in the statement of the lemma.

\end{proof}

Now we derive estimates relating the $\U{1}{2}$ coercivity constants of the Hessians of $\bqcab$ and $\nl$. Define
\begin{gather*}
c(y) = \inf_{\norm{u'}_{l^2_\eps} = 1} \fv^2 \nl (y) [u,u], \qquad
c_\beta(y) = \inf_{\norm{u'}_{l^2_\eps} = 1} \fv^2 \bqnl (y) [u,u],  \\
c_\gamma(y) = \inf_{\norm{u'}_{l^2_\eps} = 1} \fv^2 \bqce (y) [u,u].
\end{gather*}

For our {\it a~priori} error estimate, we would like to show
that if $y_a$ is a strongly stable minimizer of $\nl,$ then
$c_{\beta} (y_{a}) \gtrsim c (y_{a})$ and $c_{\gamma} (y_{a}) \gtrsim c (y_{a})$. However, such a
general result was not proved for the QNL method in
\cite{ortner:qnl1d}.  Instead, a weaker stability result that
is restricted to ``elastic states'' without
defects~\cite{Ortner:2008a} was proved. We now extend this {\it
a~priori} stability result to the BQC method.

\begin{remark}[A bound on second neighbor interactions] \label{rem: concave second nbrs}
We will place an additional condition on the set of admissible deformations
in order to prove our {\it a priori} and {\it a posteriori} stability
estimates in Theorems~\ref{thm: aprstab} and~\ref{thm: a posteriori stability}.
We will assume that
\begin{equation}\label{eqn: lower bound on y'}
\min_\xi y'_\xi \geq \frac{r^*}{2}.
\end{equation}
Under this assumption,
the constants $B_\xi$ and $\bar{B}_\xi$ in the expressions for the atomistic and BQC Hessians are nonnegative.
Assumption~\eqref{eqn: lower bound on y'} is justified since $y'_\xi
\leq \frac{r^*}{2}$ only under extreme compression, and in that case
the second nearest neighbor pair interaction model~\eqref{def: nl
  energy} itself can be expected to be invalid. The authors
of~\cite{ortner:qnl1d},~\cite{Dobson:2008b}, and~\cite{Dobson:2008c}
all consider energies similar to~\eqref{def: nl energy}, and they all
make assumptions similar to~\eqref{eqn: lower bound on y'}
(see Section~2.3 of~\cite{ortner:qnl1d} for further
discussion of this point).
\end{remark}

\begin{theorem}[A priori stability]\label{thm: aprstab}
Let $\underline{A} = \min_\xi A_\xi
(y)$. Assume that $\min_\xi y'_\xi \geq \frac{r^*}{2}>0$.

\begin{enumerate}
\item
Let $\bqce$ be the BQCE energy with blending function $\gamma$.
We have
\begin{equation*}
\begin{split}
c_\gamma(y) \geq \underline{A} - &2\Ctwo\norm{\Delta^2\alpha_\xi}_{\ell^\infty}\\
 -&\minusstabbound,
\end{split}
\end{equation*}
where $\alpha_\xi := \overline{\gamma_\xi}$,
$\beta_\xi := 1 - \frac{\gamma_{\xi+1} + \gamma_{\xi-1}}{2}$,
and $\bar{C}_i = C_i(r^{min})$,
$i = 2,3,4,$ for $r^{min}:=2\,\min_{\xi \in \mathbb{Z}} y'_\xi.$
\item
Let $\bqnl$ be the BQNL energy with blending function $\beta$.
We have
\begin{equation*}
c_\beta(y) \geq \underline{A} - \minusstabbound,
\end{equation*}
where $\bar{C}_i = C_i(r^{min})$,
$i = 3,4,$ for $r^{min}:=2\,\min_{\xi \in \mathbb{Z}} y'_\xi.$
\end{enumerate}
\end{theorem}
\begin{proof}
Let $u \in \U{1}{2}$ with $\norm{u}_{\U{1}{2}} = 1$. Recall
that by formula~\eqref{eqn: formula for Hessian} we have
\begin{equation*}
\fv^2 \bqcab (y)[u,u] = \eps \periodsum \abar_\xi |u'_\xi|^2 + \eps^2 \bar{B}_\xi |u''_\xi|^2
\end{equation*}
for any BQC energy $\bqcab$.
Since we assume $\min_\xi y'_\xi \geq \frac{r^*}{2}$,
all the coefficients $\bar{B}_\xi$ are nonnegative,
so
\begin{equation*}
\fv^2 \bqcab (y)[u,u] \geq \eps \periodsum \abar_\xi |u'_\xi|^2.
\end{equation*}

Now we assume that the energy $\bqcab$ is a BQCE energy.  Then
using the first part of Lemma~\ref{lem: difference of hessian cfcts}
we have
\begin{align}
\fv^2 \bqce (y)[u,u] &\geq \eps \periodsum A_\xi |u'_\xi|^2 - 2 \Ctwo \norm{\Delta^2\alpha_\xi}_{\ell^\infty} \notag\\
                    & \quad \quad -\minusstabbound \notag \\
                  &\geq \underline{A} -2  \Ctwo \norm{\Delta^2\alpha_\xi}_{\ell^\infty} \notag\\
                  &\quad \quad- \minusstabbound.\notag
\end{align}
This proves the first claim made in the statement of the theorem.

For the BQNL method, we make a similar estimate using the second part of
Lemma~\ref{lem: difference of hessian cfcts}. We have
\begin{align}
\fv^2 \bqnl (y)[u,u] &\geq \eps \periodsum A_\xi |u'_\xi|^2 - \minusstabbound \notag\\
                  &\geq \underline{A} - \minusstabbound.
\end{align}
This proves the second claim made in the statement of the theorem.
\end{proof}

\begin{remark}[Accuracy of critical strain]\label{rem: accuracy of critical strain}
Let $y^F$ be the uniform deformation $y^F_\xi := F\eps \xi$. Let $F^*$ be the
strain at which the lattice loses stability in the atomistic model, so
\begin{equation*}
 F^* := \inf \{F \in (0,\infty) : c(y^F) \leq 0\}.
\end{equation*}
We call $F^*$ the \emph{critical strain} of the atomistic model, and
we define the critical strains of the Cauchy-Born,
BQCE, and BQNL energies similarly. The significance of the critical
strain is discussed in~\cite{doblusort:qce.stab}.

Under uniform strain,
the $\U{-1}{2}$ coercivity constant $c_{cb}(y^F)$ of the Cauchy-Born
energy is
\begin{equation*}
c_{cb}(y^F) = \p''(F)+4\p''(2F) =: A_F.
\end{equation*}
Furthermore, by formula~\eqref{eqn: nl hessian cfcts}, for the atomistic model
under uniform strain $\underline{A} := \min_\xi A_\xi = A_F$.
Thus, by Theorem~\ref{thm: aprstab}, we see
\begin{equation}\label{eqn: critical stability}
\begin{split}
 c_\gamma(y^F) &\geq \underline{A} - 2\Ctwo\norm{\Delta^2\alpha_\xi}_{\ell^\infty}  = c_{cb}(y^F) -  2\Ctwo\norm{\Delta^2\alpha_\xi}_{\ell^\infty} \geq c_{cb}(y^F) - 2 \bar{C}_2 C_\gamma \kk^{-2}, \mbox{\quad and}\\
 c_\beta(y^F) &\geq \underline{A} = c_{cb}(y^F),
 \end{split}
\end{equation}
where $C_\gamma$ is the constant which arises in estimate~\eqref{eqn: ghost force estimate}, and $\kk$ is
the number of atoms in the interface. In the language of~\cite{doblusort:qce.stab},
estimates~\eqref{eqn: critical stability} imply that there is no error in the critical strain of the BQNL energy,
and that the error in the critical strain of the BQCE energy decreases as $\kk^{-2}$.
\end{remark}

For an {\it a posteriori} existence result,
one would like to show that if $y_{bqc}$ is a strongly stable
minimizer of $\bqcab$, then $c(y_{bqc}) \gtrsim c_\beta (y_{bqc})$.
We will not present an {\it a posteriori} existence result,
but we give an {\it a posteriori} stability estimate anyway.
Using this estimate, one can prove {\it a posteriori} existence theorems similar to Theorems~\ref{thm: a priori existence for BQNL}
and~\ref{thm: a priori existence for BQCE}.
\begin{theorem}\label{thm: a posteriori stability}
\emph{(A posteriori stability)} Assume that $\min_\xi y'_\xi
\geq \frac{r^*}{2}>0$.
\begin{enumerate}
\item
Let $\bqce$ be the BQCE energy with blending function $\gamma$.
We have
\begin{equation*}
\begin{split}
c(y) \geq c_\gamma(y) -&2\Ctwo\norm{\Delta^2\alpha_\xi}_{\ell^\infty}\\
-&\minusstabbound.
\end{split}
\end{equation*}
where $\alpha_\xi := \overline{\gamma_\xi}$,
$\beta_\xi := 1 - \frac{\gamma_{\xi+1} + \gamma_{\xi-1}}{2}$,
and $\bar{C}_i = C_i(r^{min})$,
$i = 2,3,4,$ for $r^{min}:=2\,\min_{\xi \in \mathbb{Z}} y'_\xi.$
\item
Let $\bqnl$ be the BQNL energy with blending function $\beta$.
We have
\begin{equation*}
c(y) \geq c_\beta(y) - \minusstabbound,
\end{equation*}
where $\bar{C}_i = C_i(r^{min})$,
$i = 3,4,$ for $r^{min}:=2\,\min_{\xi \in \mathbb{Z}} y'_\xi.$
\end{enumerate}
\end{theorem}
\begin{proof}
The proof of the theorem is extremely similar to the proof of
Theorem~\ref{thm: aprstab}. We use Lemma~\ref{lem: difference
of hessian cfcts} and that $B_\xi \geq 0$ whenever $\min_\xi
y'_\xi \geq \frac{r^*}{2}$.
\end{proof}

\section{A~priori error estimates for the BQCE and BQNL methods}\label{subsec: a priori existence}
We are now ready to prove our {\it a~priori} error estimates.
Our results generalize Theorem~8 and its proof given in~\cite{ortner:qnl1d}.
Before stating the theorem, it is convenient to establish some notation.
Let $y$ be a minimizer of the energy $\Phi^{total}$.
We will use Theorem~\ref{thm: aprstab} in the proof of Theorem~\ref{thm: a priori existence for BQNL}.
Thus, we must assume that $y$ is an elastic state.
That is, we assume that $\underline{A} := \min_\xi A_\xi > 0$ for $\fv^2 \Phi(y)$.
See the discussion preceding Theorem~\ref{thm: aprstab} for more explanation.
Now define the energy $\gtotal: \Y \rightarrow \Real$ by
$\gtotal(y) := \bqce(y) - <f,y>$.

We realize that the statement of Theorem~\ref{thm: a priori existence for BQCE} may seem slightly complicated,
but the basic idea of the theorem is quite simple.
In essence, Theorem~\ref{thm: a priori existence for BQCE} states that if $y$ is a stable minimizer
of the atomistic energy which is sufficiently smooth in the continuum region,
and if the ghost force error is sufficiently small,
then there exists a stable minimizer $y_{\gamma}$ of the BQCE energy which is close to $y$.
The conditions~\eqref{eqn: apr thm delta one}~and~\eqref{eqn: apr thm delta two}
make the hypotheses that $y$ must be sufficiently smooth in the continuum region
and that the ghost force must be small precise.

\begin{theorem}[A priori error estimate for the BQCE method]\label{thm: a priori existence for BQCE}
Let $\gtotal$, $y$, and $\underline{A}$ be as above.
Assume that $\underline{A} > 0$, and that $\min_\xi y'_\xi \geq \frac{r^*}{2}>0$.
  There exist
constants $\delta_1 := \delta_1(\underline{A})$ and $\delta_2
:= \delta_2(\min_\xi y'_\xi, \underline{A},\Cone, \Ctwo, \Cthr,
\gamma)$ so that if
 \begin{equation}\label{eqn: apr thm delta one}
 2\Ctwo\norm{\Delta^2\alpha_\xi}_{\ell^\infty} + 2\eps\Cthr\norm{\Delta\beta_\xi y''_\xi}_{\ell^\infty} +2 \eps^2 \{ \Cthr\norm{(1-\beta_{\xi-1})y'''_\xi}_{\ell^\infty} +\Cfour\norm{(1-\beta_{\xi})(y''_\xi)^2}_{\ell^\infty} \} \leq \delta_1,
 \end{equation}
and
 \begin{equation}\label{eqn: apr thm delta two}
  \eps^{-\half} \Cone\norm{\Delta^2\alpha_\xi}_\lte + \ltwoconsbound{\eps^{\half}}{\eps^{\frac{3}{2}}} \leq \delta_2,
 \end{equation}
then there exists a locally unique minimizer $y_\gamma$ of $\gtotal$ which satisfies
\begin{equation*}
 \norm{y - y_\gamma}_{\U{1}{2}} \leq  \frac{4}{\underline{A}} \left [\Cone\norm{\Delta^2\alpha_\xi}_\lte + \ltwoconsbound{\eps}{\eps^2}\right].
\end{equation*}
\end{theorem}

\begin{proof}
Let $\mathcal{F} : \U{1}{2} \rightarrow \U{-1}{2}$ by $\mathcal{F} (w) := \fv \gtotal (y + w)$.
We will apply Theorem~\ref{thm: ift} to $\mathcal{F}$ with $x_0 := 0$.
In order to apply Theorem~\ref{thm: ift}, we need to find a bound $\eta$ on the residual $\mathcal{F}(0)$,
a bound $\sigma$ on $\norm{(\fv \mathcal{F}(0))^{-1}}_{L(\U{-1}{2}, \U{1}{2})}$,
and a Lipschitz constant $L$ for $\fv \mathcal{F}$ on the ball $B_{\U{1}{2}}(0, 2\eta\sigma)$.
We will find $\eta$ using the modeling estimates given in Theorem~\ref{thm: modeling},
and we will find $\sigma$ using the stability estimates given in Theorem~\ref{thm: aprstab}
together with the condition~\eqref{eqn: apr thm delta one}.
Once we have these constants, we must verify the condition $2L\sigma^2\eta < 1$.
As we will see, this condition holds if $\delta_1$ and $\delta_2$ are taken sufficiently small.
Once these facts are established, Theorem~\ref{thm: ift} implies that there exists $y_\gamma \in \Y$
such that $\mathcal{F}(y_\gamma) = 0$ and $\norm{y - y_\gamma}_{\U{1}{2}} \leq 2 \eta \sigma$.
Finally, we show that $y_\gamma$ is a strongly stable minimizer of $\gtotal$ which proves Theorem~\ref{thm: a priori existence for BQNL}.

\textbf{1. Modeling error: $\norm{\mathcal{F}(0)}_{\U{-1}{2}}
\leq \eta$.} First, we find $\eta$ so that
$\norm{\mathcal{F}(0)}_{\U{-1}{2}} \leq \eta$. By
Theorem~\ref{thm: modeling}, we have
\begin{equation*}
\begin{split}
  \norm{\mathcal{F}(0)}_{\U{-1}{2}} &= \norm{\fv \bqce (y)}_{\U{-1}{2}} \\
  &\leq \Cone\norm{\Delta^2\alpha_\xi}_\lte + \ltwoconsbound{\eps}{\eps^2},
\end{split}
\end{equation*}
so we take in Theorem~\ref{thm: ift}
\begin{equation*}
\eta := \Cone\norm{\Delta^2\alpha_\xi}_\lte + \ltwoconsbound{\eps}{\eps^2}.
\end{equation*}

\textbf{2. Stability: $\norm{(\fv \mathcal{F}(0))^{-1}}_{L(\U{-1}{2},\U{1}{2})} \leq \sigma$.}
Second, we find $\sigma$ so that
\begin{equation*}
\norm{(\fv \mathcal{F}(0))^{-1}}_{L(\U{-1}{2}, \U{1}{2})} \leq \sigma.
\end{equation*}
By Theorem~\ref{thm: aprstab}, we have
\begin{equation}\label{eqn: apr thm stability}
 c_\beta(y) \geq \underline{A} -  2\Ctwo\norm{\Delta^2\alpha_\xi}_{\ell^\infty} - \minusstabbound.
\end{equation}
Then if we take $\delta_1 \leq  \frac{\underline{A}}{2}$,
we can combine~\eqref{eqn: apr thm stability}~and~\eqref{eqn: apr thm delta one} to find
\begin{equation*}
\begin{split}
c_\beta (y) &\geq \underline{A} -  2\Ctwo\norm{\Delta^2\alpha_\xi}_{\ell^\infty} - \minusstabbound\\
            &\geq \underline{A} - \delta_1 \\
            &\geq \frac{\underline{A}}{2}.
\end{split}
\end{equation*}
In that case,
\begin{equation*}
\norm{(\fv \mathcal{F}(0))^{-1}}_{L(\U{-1}{2}, \U{1}{2})} = \norm{(\fv^2 \bqce(y))^{-1}}_{L(\U{-1}{2}, \U{1}{2})} \leq \frac{1}{c_\beta(y)} \leq \frac{2}{\underline{A}}.
\end{equation*}
So we can use
\begin{equation}\label{eqn: apr thm sigma}
\sigma := \frac{2}{\underline{A}}.
\end{equation}

\textbf{3. Lipschitz bound of $\fv \mathcal{F}$ on
$B_{\U{1}{2}}(0, 2 \eta \sigma)$.} Now we bound the Lipschitz
constant of $\fv \mathcal{F}$ on $B_{\U{1}{2}}(0, 2 \eta
\sigma)$. The Lipschitz bound follows easily if we can assume
that
\begin{equation}\label{eqn: bound on y' + w'}
 y'_\xi + w'_\xi \geq \half y'_\xi
\end{equation}
for all $w$ with $\norm{w}_{\U{1}{2}} \leq 2 \eta \sigma$. We will choose $\delta_2$ so that~\eqref{eqn: apr thm delta two} implies~\eqref{eqn: bound on y' + w'}. To that end, observe that for $\norm{w}_{\U{1}{2}} \leq 2 \eta \sigma$ we have
\begin{equation*}
 \norm{w'}_{\ell^\infty} \leq \eps^{-\half} \norm{w'}_{\lte} \leq 2 \eps^{-\half} \eta \sigma = M'_1 \eps^{-\half} \eta
\end{equation*}
where $M'_1 := 2 \sigma$. Then we need only choose $\delta_2 \leq \frac{1}{2 M'_1} \min_\xi y'_\xi$ to ensure that~\eqref{eqn: bound on y' + w'} holds.

To see the Lipschitz bound, let $w_1, w_2 \in B_{\U{1}{2}}(0, 2 \eta \sigma)$.
Then expand
\begin{equation*}
|\fv^2\bqce(y+w_1)[u,v] - \fv^2\bqce(y+w_2)[u,v]|
\end{equation*}
using formula~\eqref{eqn: formula for Hessian}.
One sees easily that
\begin{equation*}
 |\fv^2\bqce(y+w_1)[u,v] - \fv^2\bqce(y+w_2)[u,v]| \leq L' \norm{w'_1 - w'_2}_{\ell^\infty} \norm{u'}_{\lte} \norm{v'}_{\lte}
\end{equation*}
where $L' := L'(C_3(\half \min_\xi y'_\xi))$. In that case, we have
 \begin{align*}
 \norm{\fv \F(w_1) - \fv \F (w_2)}_{L(\U{1}{2},\U{-1}{2})} &= \norm{\fv^2 \bqce (y + w_1) - \fv^2 \bqce (y + w_2)}_{L(\U{1}{2},\U{-1}{2})} \\
							                           &\leq \eps^{-\half} L' \norm{w_1 -w_2}_{\U{1}{2}}
							= L \norm{w_1 - w_2}_{\U{1}{2}}
\end{align*}
where $L:=\eps^{-\half} L'.$

\textbf{4. The fixed point condition: $2 L \sigma^2 \eta < 1$.}
We show that if $\delta_2$ is sufficiently small then $2 L
\sigma^2 \eta < 1$. We have
\begin{equation}
 2 L \sigma^2 \eta <  1 \Leftrightarrow 2 \eps^{-\half} L' \sigma^2 \eta < 1
                        \Leftrightarrow \eps^{-\half} \eta < \frac{\underline{A}^2}{8 L'}\label{eqn: second condition for delta 2}.
\end{equation}
Observe that~\eqref{eqn: apr thm delta two} could be restated as $\eps^{-\half} \eta \leq \delta_2$ with $\eta$ as defined in part~(1) above.
Thus, inequality~\eqref{eqn: second condition for delta 2} holds whenever $\delta_2 < \frac{\underline{A}^2}{8 L'}$.

\textbf{5. The error estimate.}
Take $\delta_1 := \frac{\underline{A}}{2}$,
and $\delta_2 < \min \{\frac{\underline{A}^2}{8 L'}, \frac{\min_\xi y'_\xi}{2 M'_1}\}$.
Then by Theorem~\ref{thm: ift} and the conclusions of parts (1-4) above,
there exists $y_\gamma \in \U{1}{2}$ with $\fv \gtotal(y_\gamma) = 0$ and
\begin{equation*}
\norm{y_\gamma - y}_{\U{1}{2}} \leq 2 \eta \sigma.
\end{equation*}

It remains only to show that the equilibrium $y_\gamma$ is a locally unique minimizer of $\gtotal$.
We will show that $\fv^2 \bqce(y_\beta)$ is positive definite.
We have
\begin{equation*}
c_\gamma(y_\gamma) \geq  c_\gamma(y) - |c_\gamma(y) - c_\gamma(y_\gamma)|
                 \geq  \frac{1}{\sigma} - L(2 \eta \sigma)
                 > \frac{1}{\sigma} - \frac{1}{\sigma} = 0.
\end{equation*}
The second inequality follows from~\eqref{eqn: apr thm sigma}, the Lipschitz bound, and the error estimate
\begin{equation*}
\norm{y_\gamma - y}_{\U{1}{2}} \leq 2 \eta \sigma.
\end{equation*}
The third inequality follows from the condition $2 L \sigma^2 \eta < 1$.
\end{proof}

We will now restate Theorem~\ref{thm: a priori existence for BQCE}
using the estimates given in Remarks~\ref{rem: ghost force} and~\ref{rem: coupling error}.
For simplicity, assume that all transitions between the atomistic and continuum models occur over regions which contain $\kk$ atoms.
Let $\alpha$ and $\beta$ be the blending functions so that $\Phi^{qce}_\gamma = \Phi_{\alpha,\beta}$.
(See~\eqref{eqn: alpha and beta for BQCE} for the precise definitions of $\alpha$ and $\beta$ given $\gamma$.)
Using the estimates given in Remarks~\ref{rem: ghost force} and~\ref{rem: coupling error},
let $C_1^\gamma$ be a constant so that
\begin{equation*}
\eps \norm{\Delta\beta_\xi}_\lte \leq C_1^\gamma \eps^{\frac{3}{2}}\kk^{-\half}
 \mbox{ and } \eps \norm{\Delta\beta_\xi}_{\ell^\infty} \leq C_1^\gamma \eps \kk^{-1},
\end{equation*}
and let $C_2^\gamma$ be a constant so that
\begin{equation*}
\norm{\Delta^2 \alpha_\xi}_\lte \leq C_2^\gamma \eps^{\half}\kk^{-\frac{3}{2}}
 \mbox{ and } \norm{\Delta^2 \alpha_\xi}_{\ell^\infty} \leq C_2^\gamma \kk^{-2}.
\end{equation*}
In essence, Corollary~\ref{cor: error estimate for BQCE} states that if $y$ is a stable minimizer of the atomistic energy
which is sufficiently smooth in the continuum region
and if the ghost force error is small,
then there exists a minimizer $y_\gamma$ of the BQCE energy $\gtotal$ which is close to $y$.
\begin{corollary}\label{cor: error estimate for BQCE}
Let $\gtotal$, $y$, and $\underline{A}$ be as in Theorem~\ref{thm: a priori existence for BQCE},
and let $C_1^\gamma$, $C_2^\gamma$, and $\kk$ be as above.
There exist
constants $\delta_1$ and $\delta_2$ so that if
 \begin{equation*}
 2 \kk^{-2} C_2^\gamma \Ctwo  + 2 \eps \kk^{-1} C_1^\gamma \Cthr \norm{y''_\xi}_{\ell^\infty(\I)} + 2 \eps^2 \{ \Cthr\norm{(1-\beta_{\xi-1})y'''_\xi}_{\ell^\infty} +\Cfour\norm{(1-\beta_{\xi})(y''_\xi)^2}_{\ell^\infty} \} \leq \delta_1,
 \end{equation*}
and
 \begin{equation*}
\kk^{-\frac{3}{2}} C_2^\gamma \Cone + \eps \kk^{-\half} C_1^\gamma \Ctwo \norm{y''_\xi}_{\ell^\infty(\I)}+ \eps^{\frac{3}{2}} \left \{ \Ctwo\norm{(1-\beta_{\xi-1})y'''_\xi}_{\lte} + \Cthr\norm{(1-\beta_{\xi})(y''_\xi)^2}_{\lte} \right \} \leq \delta_2,
 \end{equation*}
then there exists a locally unique minimizer $y_\gamma$ of $\gtotal$ which satisfies
\begin{align*}
 &\norm{y - y_\gamma}_{\U{1}{2}} \\
 &
 \leq  \frac{4}{\underline{A}} \left [\eps^{\half}\kk^{-\frac{3}{2}} C_2^\gamma \Cone +
 \eps^{\frac{3}{2}} \kk^{-\half} C_1^\gamma \Ctwo \norm{y''_\xi}_{\ell^\infty(\I)}
 + \eps^{2} \left \{ \Ctwo\norm{(1-\beta_{\xi-1})y'''_\xi}_{\lte} + \Cthr\norm{(1-\beta_{\xi})(y''_\xi)^2}_{\lte} \right \}\right].
\end{align*}
\end{corollary}

We give a similar {\it a priori} error estimate for the BQNL method in Theorem~\ref{thm: a priori existence for BQNL}.
Following the notation adopted in Theorem~\ref{thm: a priori existence for BQCE}, let $\btotal$ be the
energy given by $\btotal(y):= \bqnl(y) - <f,y>$. Let $y$ be a minimizer of the energy $\Phi^{total}$
such that $\underline{A} := \min_\xi A_\xi > 0$ for $\fv^2 \Phi(y)$.

\begin{theorem}[A priori error estimate for the BQNL method]\label{thm: a priori existence for BQNL}
Let $\Phi^{total}_\beta$, $y$, and $\underline{A}$ be as above.
Assume that $\underline{A} > 0$, and that $\min_\xi y'_\xi \geq \frac{r^*}{2}>0$. There exist
constants $\delta_1 := \delta_1(\underline{A})$ and $\delta_2
:= \delta_2(\min_\xi y'_\xi, \underline{A}, \Ctwo, \Cthr,
\beta)$ so that if
 \begin{align*}
  \stabbound &\leq \delta_1, \mbox{\quad and}   \\
   \ltwoconsbound{\eps^{\half}}{\eps^\frac{3}{2}} &\leq \delta_2,
 \end{align*}
then there exists a locally unique minimizer $y_\beta$ of $\btotal$ which satisfies
\begin{equation*}
 \norm{y - y_\beta}_{\U{1}{2}} \leq \frac{4}{\underline{A}} \left [ \ltwoconsbound{\eps}{\eps^2} \right ].
\end{equation*}
\end{theorem}

\begin{proof}
The proof of Theorem~\ref{thm: a priori existence for BQNL} is similar to the the proof of Theorem~\ref{thm: a priori existence for BQCE};
we simply use the modeling and stability estimates for BQNL in place of those for BQCE.
\end{proof}

\section{Conclusion}
We have proposed a smoothly blended version of the
quasicontinuum energy (QCE) which we call the blended
quasicontinuum energy (BQCE). BQCE blends the atomistic and
corresponding Cauchy-Born continuum energies used in the QCE
method over an interfacial region whose thickness is a small
number $k$ of blended atoms. We analyze the accuracy of BQCE
applied to the problem of a one-dimensional chain with
next-nearest neighbor interactions. For this test problem, we
show how to choose the optimal blending function for weighting
the atomistic and Cauchy-Born continuum energies. If BQCE is
implemented using the optimal blending function, the critical
strain error of BQCE can be reduced by a factor of $k^2$. Thus,
we believe that BQCE could be used to accurately compute the
deformation of crystalline solids up to lattice instabilities.

We are continuing the development of the BQCE energy by
modifying the code developed to study the accuracy
 of quasicontinuum methods for two benchmark problems
 --- the stability of a Lomer dislocation pair under shear and
 the stability of a lattice to plastic slip under tensile
 loading~\cite{luskin.durham}.  We note that the potential
to significantly improve the accuracy of existing
quasicontinuum codes by easily implemented modifications is a
very desirable feature of the BQCE approach.  We think our
theoretical analysis of the accuracy
 near instabilities for one-dimensional model problems can
 successfully explain most of the computational results for
 these multi-dimensional benchmark problems.

We expect that some form of the stability and error estimates
in this paper can be generalized to atomistic models with
finite range interactions by extending the techniques given in
\cite{LuskinXingjie}. We are investigating the extension
of our one-dimensional analysis to the multi-dimensional
setting, but we expect that any multi-dimensional analysis
would likely be restricted to perturbations from the
ground state which are far from lattice instability.
We will thus need to rely on our one-dimensional analysis
to attempt to understand the computational results from
our multi-dimensional benchmark studies~\cite{luskin.durham}.

\section{Acknowledgments}
We acknowledge the useful suggestions of Xingjie Li and Dr. Christoph Ortner.

\end{document}